\documentclass[a4paper]{amsart}
\usepackage{amssymb, enumitem}
\usepackage[all]{xy}
\usepackage{hyperref, aliascnt}
\usepackage{mathtools}
\usepackage{bbm}
\def\today{\number\day\space\ifcase\month\or   January\or February\or
   March\or April\or May\or June\or   July\or August\or September\or
   October\or November\or December\fi\   \number\year}

\theoremstyle{definition}

\newaliascnt{thmCt}{lma}
\newtheorem{thm}[thmCt]{Theorem}
\aliascntresetthe{thmCt}

\newaliascnt{corCt}{lma}
\newtheorem{cor}[corCt]{Corollary}
\aliascntresetthe{corCt}

\newaliascnt{propCt}{lma}
\newtheorem{prop}[propCt]{Proposition}
\aliascntresetthe{propCt}

\newtheorem*{thm*}{Theorem}
\newtheorem*{cor*}{Corollary}
\newtheorem*{prop*}{Proposition}

\newcounter{theoremintro}
\newtheorem{thmintro}[theoremintro]{Theorem}

\newtheorem{egintro}[theoremintro]{Example}

\newaliascnt{pgrCt}{lma}

\aliascntresetthe{pgrCt}

\newaliascnt{dfCt}{lma}
\newtheorem{df}[dfCt]{Definition}
\aliascntresetthe{dfCt}

\newaliascnt{remCt}{lma}
\newtheorem{rem}[remCt]{Remark}
\aliascntresetthe{remCt}

\newaliascnt{remsCt}{lma}

\aliascntresetthe{remsCt}

\newaliascnt{egCt}{lma}
\newtheorem{eg}[egCt]{Example}
\aliascntresetthe{egCt}

\newaliascnt{egsCt}{lma}

\aliascntresetthe{egsCt}

\newaliascnt{qstCt}{lma}

\aliascntresetthe{qstCt}

\newaliascnt{pbmCt}{lma}

\aliascntresetthe{pbmCt}

\newaliascnt{notaCt}{lma}
\newtheorem{nota}[notaCt]{Notation}
\aliascntresetthe{notaCt}

\newaliascnt{cnjCt}{lma}

\aliascntresetthe{cnjCt}

\newcommand{\beq}{\begin{equation}}
\newcommand{\eeq}{\end{equation}}
\newcommand{\beqa}{\begin{eqnarray*}}
\newcommand{\eeqa}{\end{eqnarray*}}
\newcommand{\bal}{\begin{align*}}
\newcommand{\eal}{\end{align*}}
\newcommand{\bi}{\begin{itemize}}
\newcommand{\ei}{\end{itemize}}
\newcommand{\be}{\begin{enumerate}}
\newcommand{\ee}{\end{enumerate}}

\newcommand{\ep}{\varepsilon}

\newcommand{\Q}{{\mathbb{Q}}}

\newcommand{\Z}{{\mathbb{Z}}}

\newcommand{\C}{{\mathbb{C}}}
\newcommand{\N}{{\mathbb{N}}}

\newcommand{\K}{{\mathcal{K}}}

\newcommand{\B}{{\mathcal{B}}}
\newcommand{\U}{{\mathcal{U}}}

\newcommand{\T}{{\mathbb{T}}}

\newcommand{\Ot}{{\mathcal{O}_2}}
\newcommand{\OI}{{\mathcal{O}_{\I}}}

\newcommand{\id}{{\mathrm{id}}}

\newcommand{\diag}{{\mathrm{diag}}}

\newcommand{\Aut}{{\mathrm{Aut}}}

\newcommand{\Ext}{{\mathrm{Ext}}}

\newcommand{\Ad}{{\mathrm{Ad}}}

\newcommand{\ca}{$C^*$-algebra}

\newcommand{\uca}{unital $C^*$-algebra}

\newcommand{\Rp}{Rokhlin property}

\newcommand{\I}{\infty}

\title[KK-theory of circle actions]{KK-theory of circle actions with the Rokhlin property}

\author{Eusebio Gardella}

\date{\today}

\thanks{\emph{Supported by}: 
  US National Science Foundation grant
DMS-1101742; the Deutsche Forschungsgemeinschaft
(through SFB 878 and an eigene Stelle); and the Humboldt foundation.}

\address{Eusebio Gardella
Mathematisches Institut, Fachbereich Mathematik und Informatik der
Universit\"at M\"unster, Einsteinstrasse 62, 48149 M\"unster, Germany.}
\email{gardella@uni-muenster.de}
\urladdr{www.math.uni-muenster.de/u/gardella/}


\begin{document}
\begin{abstract} 
We investigate the structure of circle actions with the Rokhlin property, particularly in relation
to equivariant $KK$-theory. Our main results are
$\mathbb{T}$-equivariant versions of celebrated results of 
Kirchberg:  
any Rokhlin action on a separable, nuclear \ca\ is 
$KK^\mathbb{T}$-equivalent to a Rokhlin action on a Kirchberg algebra; any Rokhlin action on an
exact separable \ca\ embeds equivariantly into 
$\mathcal{O}_2$ (with its unique Rokhlin action);
and two circle actions with the Rokhlin
property on a Kirchberg algebra are conjugate if
and only if they are $KK^\mathbb{T}$-equivalent. 

In the presence of the UCT, $KK^\mathbb{T}$-equivalence
for Rokhlin actions reduces to isomorphism of a 
$K$-theoretical invariant, namely of a canonical 
pure extension naturally associated to any Rokhlin
action, 
and we provide a complete 
description of the extensions that arise
from actions on nuclear $C^*$-algebras.
In contrast with the non-equivariant setting,
an isomorphism between the $K^\mathbb{T}$-theories of Rokhlin actions on Kirchberg algebras does not 
necessarily lift to a $KK^\mathbb{T}$-equivalence.
 \end{abstract}

\maketitle
 
\vspace{-0.7cm}
\renewcommand*{\thetheoremintro}{\Alph{theoremintro}}
\section{Introduction}

Crossed products have provided some of the most relevant examples in the theory of $C^*$-algebras, and the study of their structure and classification is 
a very active field of research. Moreover, the classification
of group actions has a long history within the theory of operator algebras.
For example, Connes' classification of automorphisms of the hyperfinite II$_1$-factor $\mathcal{R}$
was instrumental in his award-winning classification of amenable factors
\cite{Con_classification_1976}.
Connes' success motivated significant efforts towards classifying
amenable group actions on hyperfinite factors, which culminated almost two decades
later.

By comparison to the von Neumann algebra setting, 
the classification of C*-dynamical systems 
is a far less developed field of research. A major
difficulty is the complicated behavior that finite order automorphisms exhibit at 
the level of $K$-theory. Even in the absence of torsion in the acting group, the 
induced action on the trace space may be wild. 
These difficulties resulted in a 
somewhat scattered collection of results, 
and a lack of a systematic approach.
A notable exception is the fruitful analysis of the Rokhlin property in \ca s, 
as is seen in the works of Kishimoto \cite{Kis_rohlin_1995,Kis_rohlin_1996},
Matui \cite{Mat_ZNactions_2011}, 
Nakamura \cite{Nak_aperiodic_2000}, 
Izumi \cite{Izu_RpI_2004, Izu_RpII_2004}, Sato 
\cite{Sat_rohlin_2010}, Nawata \cite{Naw_finite_2016}, 
and the author and Santiago
\cite{GarSan_equivariant_2016}, to mention a few. 


The advances in Elliott's classification programme
(which is by now essentially complete; see \cite{EllGonLinNiu_classification_2015,TikWhiWin_quasidiagonality_2017})
suggest that group actions on purely infinite
\ca s may be more accessible, and there already
exist very encouraging results in this direction.
In \cite{Sza_equivariant_2018}, Szabo proved 
versions of Kirchberg's absoprtion results 
$\Ot\otimes A\cong \Ot$ and $\OI\otimes A\cong A$
(with suitable $A$) for outer actions of 
amenable groups, using actions on $\Ot$ and $\OI$ which
have a suitable version of the 
Rokhlin property. (An equivariant
version of the absorption result
$\Ot\otimes A\cong \Ot$ was proved for 
\emph{exact} groups by Suzuki \cite{Suz_equivariant_2020}.)
More recently,
Meyer \cite{Mey_classification_2019} began 
exploring the classification of 
actions of torsion-free amenable groups using 
$KK^G$-theory, particularly in what
refers to lifting an isomorphism between $K$-theoretical data to a $KK^G$-equivalence. A recurrent issue in this setting
is that satisfactory results can only be expected
if either the action has some variation of the 
Rokhlin property, or the group is torsion-free.


In this work, we use equivariant $KK$-theory to
study circle actions with the Rokhlin property,
and obtain $\T$-equivariant versions of celebrated
results of Kirchberg concerning 
simple, purely infinite, separable, nuclear
\ca s (also known as Kirchberg algebras);
see Theorems~\ref{thmintro:Classif}, \ref{thmintro:KKeqKir} and \ref{thmintro:Emb} below. 
By comparison to the \emph{continuous} 
Rokhlin property
(studied in \cite{AraKub_compact_2017} and \cite{Gar_Kir2}), Rokhlin actions are a much
richer class with less rigid behavior. 
Accordingly, more involved arguments are needed
in this setting.

The main reason to focus on circle actions is that the 
combination of one-dimensio\-na\-li\-ty with 
the fact that $\T$ is a Lie group produces 
phenomena that cannot be expected beyond this setting.
An example of this, which is crucial to our work and already fails for $\T^2$, is the existence of a predual automorphism: 





\begin{thmintro}
(See \autoref{thm:RpisDual}) Let $A$ be a unital \ca, and let
$\alpha\colon \T\to\Aut(A)$ be an action with the Rokhlin property. Then
there exists $\check{\alpha}\in\Aut(A^\alpha)$ such that $\alpha$ is 
conjugate to the dual action $\widehat{\check{\alpha}}$.
\end{thmintro}

We completely characterize the automorphisms
that arise as preduals of circle actions with the Rokhlin
property; see \autoref{thm:DualityRpAr}. Using this,
we show that every circle action with the \Rp\ has a 
naturally associated PExt-class. 

\begin{thmintro}\label{thmintro:PExt}
(See \autoref{thm:PureExt})
Let $A$ be a unital \ca, and let
$\alpha\colon \T\to\Aut(A)$ be a Rokhlin action. 
Then the natural map $K_0(A^\alpha)\hookrightarrow K_0(A)$ is an order-embedding, and there is a canonical \emph{pure} extension $\Ext_\ast(\alpha)$ 
given by
\[\xymatrix{0\ar[r] & K_\ast(A^\alpha)\ar[r] &K_\ast(A)\ar[r] & K_\ast(SA^\alpha)\ar[r]& 0}.\]
\end{thmintro}

The fact that the above extension is pure is far
from obvious, and it ultimately depends on the 
existence of a sequence of ucp maps $A\to A^\alpha$
which are asymptotically multiplicative and 
asymptotically the identity on $A^\alpha$. 

Our most interesting results are related to 
equivariant $KK^\T$-theory in the setting of Kirchberg algebras. 
When $A$ is a Kirchberg algebra, we 
show that so is $A^\alpha$ and that $\check{\alpha}$
is aperiodic (\autoref{prop:KirAperiodic}), which
gives us access to Nakamura's work \cite{Nak_aperiodic_2000}. 
We use this to obtain a $\T$-equivariant 
version of the Kirchberg-Phillips classification theorem
for actions with the Rokhlin property:

\begin{thmintro}\label{thmintro:Classif}
(See \autoref{thm:ClassRpKir})
Let $\alpha\colon \T\to\Aut(A)$ and $\beta\colon\T\to\Aut(B)$ be actions
on unital Kirchberg algebras with the Rokhlin property. Then $\alpha$ and 
$\beta$ are conjugate if and only if they are
unitally $KK^\T$-equivalent. In the presence of the 
UCT, this is in turn equivalent to the existence of a graded
isomorphism $\Ext_\ast(\alpha)\cong\Ext_\ast(\beta)$, which preserves
unit classes and is compatible with suspension shifts (see \autoref{df:IsomExt}).
\end{thmintro}

Unlike in Kirchberg-Phillips' classification, 
in the presence of the UCT it does not suffice 
to assume
that both actions have isomorphic $K^\T$-theory
(this is a big difference with the case of the 
continuous Rokhlin property \cite{AraKub_compact_2017, Gar_Kir2}).
It also does not suffice for the actions to 
have 
isomorphic Meyer's $L$-invariant 
$L_\ast^\T(A,\alpha)=K_\ast^\T(A,\alpha)\oplus K_\ast(A)$ (\cite{Mey_classification_2019}):

\begin{egintro}\label{egintro}(See \autoref{eg:EqKThyNotEnough}.)
There exist a UCT Kirchberg algebra $A$ and Rokhlin actions $\alpha,\beta\colon \T\to\Aut(A)$ such that $K_\ast^\T(A,\alpha)\cong K_\ast^\T(A,\beta)$ as $R(\T)$-modules, although $\alpha$ and $\beta$
are not $KK^\T$-equivalent. 
One can even construct the actions
so that $A^\alpha\cong A^\beta$ and $A\rtimes_\alpha\T\cong A\rtimes_\beta\T$, all satisfying the UCT.
\end{egintro}

The above example shows an interesting phenomenon, 
which we put into perspective. In Example~10.6 
in~\cite{RosSch_kunneth_1986}, Rosenberg and Schochet 
construct two circle actions
on commutative \ca s with isomorphic $K^\T$-theory, 
which are not $KK^\T$-equivalent. In their example, 
the underlying algebras are not even $KK$-equivalent, 
so the actions cannot be $KK^\T$-equivalent. 
As communicated to us by Claude Schochet,
Example~\ref{egintro} is the first construction 
of two circle actions on 
\emph{the same} \ca, satisfying the UCT, with
isomorphic fixed point algebras and 
crossed products, which all satisfy the UCT, 
and isomorphic $K^\T$-theory, that are 
not $KK^\T$-equivalent. 

We also obtain a range result in the 
context of Theorem~\ref{thmintro:Classif}, showing
that the only $K$-theoretic obstructions are the 
ones obtained in \autoref{thm:rangeInv}. 
Roughly speaking, for any \emph{pure} extension 
\[\tag{$\mathcal{E}_\ast$} \ \ \ 0\to K_\ast \to G_\ast \to K_{\ast+1}\to 0,\] 
there exists a circle action $\alpha\colon\T\to\Aut(A)$ on a UCT Kirchberg algebra $A$ such that
$\Ext_\ast(\alpha)$ is isomorphic to $\mathcal{E}_\ast$.

With a classification of Rokhlin actions on Kirchberg
algebras in terms of $KK^\T$-theory at our disposal, 
it is natural to ask which Rokhlin actions are 
$KK^\T$-equivalent to a Rokhlin action on a Kirchberg
algebra. As it turns out, Rokhlin actions on Kirchberg
algebras represent all separable, nuclear $KK^\T$-classes of Rokhlin actions:

\begin{thmintro}\label{thmintro:KKeqKir}
(See \autoref{thm:EveryRpKKequivKirch}.)
Let $A$ be a separable, nuclear, \uca, and let
$\alpha\colon \T\to\Aut(A)$ have the \Rp. 
Then there exist a unique unital Kirchberg algebra $D$ 
and a unique circle action $\delta\colon \T\to\Aut(D)$
with the \Rp\ such that $(A,\alpha)\sim_{KK^\T}(D,\delta)$ unitally.
\end{thmintro}

The theorem above cannot be extended to 
actions $\alpha$ that do not necessarily have 
the Rokhlin property, since there are 
obstructions to being $KK^\T$-equivalent 
to a Rokhlin action 
(for example as in Theorem~\ref{thmintro:PExt}).
Theorem~\ref{thmintro:KKeqKir} 
should be compared to
Theorem~2.1 of~\cite{Mey_classification_2019}, 
where Meyer
shows that every circle action on a separable, nuclear
\ca\ is $KK^\T$-equivalent to an outer action on a 
Kirchberg algebra. 
It is not clear from Meyer's construction 
that the resulting action on the 
Kirchberg algebra has the Rokhlin property if the original
one does.
We therefore could not adapt his argument to
our context, and instead use older ideas of Kirchberg, 
applied at the level of the predual $\check{\alpha}$. 
A key ingredient in the proof is the following 
$\T$-equivariant version of Kirchberg's $\Ot$-embedding 
theorem. 
 
\begin{thmintro}\label{thmintro:Emb} (See \autoref{thm:EmbedExact}.)
Let $A$ be a separable, exact, unital \ca, and let 
$\alpha\colon\T\to\Aut(A)$ have the \Rp.
Then there is a unital, equivariant embedding 
$(A,\alpha) \hookrightarrow (\Ot,\gamma)$,
where $\gamma$ is the unique Rokhlin action on $\Ot$. 
\end{thmintro}


The results here presented are an expanded version 
of Chapter~IX of my PhD thesis \cite{Gar_thesis_2015}. Since the first preprint version of this work appeared
on the arxiv, Arano and Kubota generalized the first
part of Theorem~C to compact groups other than $\T$;
see Proposition~4.8 in~\cite{AraKub_compact_2017}.
The methods are quite different, since we take full
advantage of the existence of a predual automorphism.
Moreover, for circle actions (as opposed to general
compact group actions), $KK^\T$-equivalence can be 
detected via a $K$-theoretical invariant in the 
presence of the UCT, and the range of this invariant 
can be completely
described. This makes the classification of circle
actions with the Rokhlin property comparatively more
accessible than that of general compact groups.


\vspace{0.2cm}

\textbf{Acknowledgements.} 
%
The author is grateful to 
a number of people for helpful 
discussions, correspondence, or feedback, including
Claude Schochet, Rasmus Bentmann, Martino Lupini, Ralf Meyer, Chris Phillips, and Hannes Thiel.

\section{Duality for circle actions with the Rokhlin property}
In this section, we study the Rokhlin property for circle actions
in connection to duality. There are two main results in this section.
First, it is shown that every circle action with the Rokhlin
property is a dual action, that is, there is an automorphism
of the fixed point algebra whose dual action is conjugate to the 
given one; see \autoref{thm:RpisDual}. 
Such an automorphism is essentially unique, and is called the 
\emph{predual} automorphism.
Second, we characterize those automorphisms that 
are predual to a circle action with the Rokhlin property; 
see \autoref{thm:DualityRpAr}. 

We begin by recalling the definition of the \Rp\ for a circle action (Definition~3.2 in~\cite{HirWin_Rp}) in a way that is useful for our purposes.

\begin{df}\label{def RP} 
Let $A$ be a \uca\ and let $\alpha\colon \T\to\Aut(A)$ be an 
action. We say that $\alpha$ has the \emph{Rokhlin property} if for every $\varepsilon>0$ and
every compact subset $F\subseteq A$, there exists a 
unitary $u\in \U(A)$ such that
\be
\item[(a)] $\|\alpha_ z(u)- z u\|<\varepsilon$ for all $ z\in\T$.
\item[(b)] $\|ua-au\|<\varepsilon$ for all $a\in F$.
\ee
\end{df}

Next, we show that we can replace the unitary $u$
in the above definition by a nearby unitary 
which satisfies condition (a) exactly.
In the terminology of \cite{Phi_equivSP}, the following shows that
the action of $\T$ on $C(\T)$ is equivariantly
semiprojective. 
 
\begin{prop}\label{prop:EqSemiProj} 
For every $\ep>0$ there exists $\delta>0$ with the following
property: whenever $\alpha\colon \T\to\Aut(A)$ is an action on a
\uca\ $A$ and $u\in \U(A)$ is a unitary 
satisfying $\|\alpha_ z(u)- z u\|<\delta$ for all $ z\in\T$, 
then there exists a unitary $v\in \U(A)$ with $\|u-v\|<\ep$ and 
$\alpha_ z(v)= z v$ for all $ z\in\T$.
\end{prop}
\begin{proof} 
Given $\ep>0$, choose $\delta<\frac{1}{3}$ small enough so that
\[\label{eqn:epSmall}\tag{2.1}\frac{2\delta}{\sqrt{1-2\delta}}+\delta<\ep.\]
Let $\alpha\colon \T\to\Aut(A)$ and $u\in\U(A)$ be as in
the statement.
Set
$x =\int_{\T}\overline{ z}\alpha_ z(u)\ d z\in A$.
Then $\|x\|\leq 1$ and $\|x-u\|\leq \delta$. One checks that 
$\|x^*x-1\|\leq 2\delta< 1$, so $x^*x$ is invertible. Moreover,
\[\label{eqn:xxstar}\tag{2.2}
\big\|(x^*x)^{-1}\big\|\leq \frac{1}{1-\|1-x^*x\|}\leq\frac{1}{1-2\delta}.\]

Set $u= x(x^*x)^{-\frac{1}{2}}$. Then $u$ is a unitary in $A$. Using that $\|x\|\leq 1$ at the first step,
and that $0 \leq 1- (x^*x)^{\frac{1}{2}} \leq  1- x^*x$ at the second step, we get
\begin{align*} \|u-x\| &\leq \big\|(x^*x)^{-\frac{1}{2}}-1\big\|
\leq \big\|(x^*x)^{-\frac{1}{2}}\big\| \big\|1-(x^*x)^{\frac{1}{2}}\big\|\\
&\stackrel{(\ref{eqn:xxstar})}{\leq} \frac{1}{\sqrt{1-2\delta}} \|1-x^*x\| \leq \frac{2\delta}{\sqrt{1-2\delta}}.
\end{align*}
Thus 
\[\|u-v\|\leq \|u-x\|+\|x-v\|\leq 
\frac{2\delta}{\sqrt{1-2\delta}}+\delta\stackrel{(\ref{eqn:epSmall})}{<} \ep.\]
For $ z\in \T$, we have 
\[\alpha_ z(x)=\int_\T \overline{\omega}\alpha_{ z\omega}(u)d\omega=
\int_\T  z\overline{\omega}\alpha_{\omega}(u)d\omega= z x.
\]
It follows that $\alpha_ z(x^*x)=x^*x$ and hence
$\alpha_ z(u)=\alpha_ z\big(x(x^*x)^{-\frac{1}{2}}\big)= z u$,
for all $ z\in\T$, so $u$ satisfies the condition in the statement.
\end{proof}

It follows from \autoref{prop:EqSemiProj} that condition~(1) in
\autoref{def RP} can be replaced with $\alpha_ z(u)= z u$
for all $ z\in\T$. 

\begin{thm}\label{thm:RpisDual} Let $A$ be a \uca\ and let $\alpha\colon\T\to\Aut(A)$ be an action with the \Rp. 
\be\item There exists an
automorphism $\theta\in \Aut(A^\alpha)$ such that 
$(A^\alpha\rtimes_\theta\Z,\widehat{\theta})$ is conjugate to $(A,\alpha)$.
\item If $\theta'\in \Aut(A^\alpha)$ is another automorphism for which
$(A^\alpha\rtimes_{\theta'}\Z,\widehat{\theta'})$ is conjugate to $(A,\alpha)$, then there is a
unitary $w\in A^\alpha$ such that $\theta=\Ad(w)\circ\theta'$.
\ee
\end{thm}
\begin{proof}
(1). Using \autoref{prop:EqSemiProj}, let $u\in\U(A)$ be a unitary satisfying
$\alpha_z (u)= z  u$ for all $ z \in\T$. 
For $a\in A^\alpha$, we have $\alpha_z (uau^*)=uau^*$ for all $ z \in\T$, and thus conjugation by $u$ determines an automorphism $\theta$ of
$A^\alpha$. Let $v\in A^\alpha\rtimes_\theta\Z$ denote the canonical unitary
implementing $\theta$. Since the pair $(\id_{A^\alpha}, u)$ is a covariant representation of $(A^\alpha,\theta)$ on $A$, there is a unique homomorphism $\varphi\colon A^\alpha\rtimes_\theta\Z\to A$
satisfying $\varphi(a)=a$ for all $a\in A^\alpha$ and $\varphi(v)=u$.

We claim that $\varphi$ is an equivariant isomorphism. Equivariance of $\varphi$ is clear, since for all $ z \in\T$ we have $\widehat{\theta}_z (a)=a$ for all $a\in A^\alpha$ and $\widehat{\theta}_z (v)= z  v$.
Injectivity of $\varphi$ follows from the fact that $\id_{A^\alpha}$ is 
injective (and that $\T$ is amenable). 
Surjectivity can be deduced using spectral subspaces, as follows.
Given $n\in\Z$, we set 
\[A_n=\{a\in A\colon \alpha_z (a)= z ^n a\mbox{ for all }  z \in\T\},
\]
which is a closed subspace of $A$. It is well-known that $\sum_{n\in\Z}A_n$ 
is dense in $A$; see, for example, part~(ix) of~Theorem~8.1.4 in~\cite{Ped_algebras_1979}. 
Note that $A_0=A^\alpha$ and that 
$u$ belongs to $A_1$. Moreover, using that $u$ is a unitary, it is easy to see that $A_n=u^nA_0$ for all $n\in\Z$. In particular, $A$ is generated as a 
\ca\ by $A_0$ and $u$. Since $A_0\cup\{u\}$ is contained in the image of $\varphi$, 
we conclude that $\varphi$ is surjective.

(2). Let $\theta'$ be as in the statement, and let $\varphi\colon (A^\alpha\rtimes_{\theta}\Z,\widehat{\theta})\to (A,\alpha)$ and $\varphi'\colon (A^\alpha\rtimes_{\theta'}\Z,\widehat{\theta'})\to (A,\alpha)$ be equivariant
isomorphisms. Let $v$ be the canonical
unitary in $A^\alpha\rtimes_\theta\Z$ that implements $\widehat{\theta}$, and let $v'$ be the canonical unitary in $A^\alpha\rtimes_{\theta'}\Z$
that implements $\widehat{\theta'}$. Set $w=\varphi(v)\varphi'(v')^*$, which is a unitary in $A$. We claim that $w$ is fixed by $\alpha$. 
For $ z\in\T$, we use equivariance of $\varphi$ and $\varphi'$ to get
$$\alpha_ z(w)=\varphi(\widehat{\theta}_ z(v))\varphi'(\widehat{\theta'}_ z(v'))^*)= z\varphi(v)\overline{ z}\varphi'(v')=w.$$
Finally, given $a\in A^\alpha$, we have
\begin{align*}(\Ad(w)\circ\theta)(a)&=(\varphi(v)\varphi'(v')^*)(\varphi'(v')a\varphi'(v')^*)(\varphi'(v')\varphi(v)^*)\\
 &=\varphi(v)a\varphi(v)^*=\theta'(a).\qedhere
\end{align*}\end{proof}

\begin{df}\label{df:predual}
Let $\alpha\colon\T\to\Aut(A)$ be an action with the \Rp.
In view of \autoref{thm:RpisDual}, we denote by $\check{\alpha}\in\Aut(A^\alpha)$ the unique automorphism for which $\widehat{\check{\alpha}}$ is conjugate
to $\alpha$. We call $\check{\alpha}$ the \emph{predual automorphism} of 
$\alpha$.
\end{df}

\begin{cor}\label{cor: FixedPtCP} Let $A$ be a \ca\ and let $\alpha\colon\T\to\Aut(A)$ be an action with the \Rp. Then there is a natural isomorphism
$A\rtimes_\alpha\T\cong A^\alpha\otimes\K(L^2(\T))$.\end{cor}
\begin{proof} This is an immediate from\autoref{thm:RpisDual} and
Takai duality. \end{proof}

Obtaining a characterization 
of those automorphisms that arise as preduals of Rokhlin actions
as in \autoref{df:predual} will be a critical
tool in the rest of this work.
Such a characterization is obtained in \autoref{thm:DualityRpAr},
using the following notion:

\begin{df}\label{df:AppRep} Let $B$ be a \ca\ and let $\beta$ be an automorphism of $B$. Then $\beta$ is said to be \emph{approximately representable} if for every
finite subset $F\subseteq B$ and every $\ep>0$, there exists a contraction
$v\in B$ satisfying
\bi\item[(a)] $\|v^*v-vv^*\|<\ep$;
\item[(b)] $\|v^*vb-b\|<\ep$ for all $b\in F$;
\item[(c)] $\|\beta(v)-v\|<\ep$; and
\item[(d)] $\|\beta(b)-vbv^*\|<\ep$ for all $b\in F$.
\ei
\end{df}

Using functional calculus, it is clear that the contraction $v$ in the above
definition can be chosen to be a unitary whenever $B$ is unital. In particular, 
approximately representable automorphisms of unital \ca s are approximately
inner. 



\begin{rem}\label{rem:ConvProduct}
We endow $\T$ with its Haar probability measure.
For an action $\alpha\colon\T\to\Aut(A)$ 
on a \uca\ $A$,
we endow $L^1(\T,A)$ with the usual $L^1$-norm
$\|\cdot\|_1$
and the operations of
twisted convolution and involution
\[(\xi\ast\eta)(z)=\int_\T \xi(\omega)\alpha_\omega(\eta(\omega^{-1}z))\ d\omega \ \ \mbox{ and } \ \ \xi^*(z)=\alpha_z(\xi(\overline{z})^*)\]
for all $\xi,\eta\in L^1(\T,A)$ and all $z\in\T$. 
Then $L^1(\T,A)$ is a dense
$\ast$-subalgebra of $A\rtimes_\alpha\T$, and the 
canonical inclusion is contractive with respect to
the $L^1$-norm on $L^1(\T,A)$ and the $C^*$-norm
on $A\rtimes_\alpha\T$. 
Recall that there is a canonical nondegenerate 
inclusion $C^*(\T)\subseteq A\rtimes_\alpha \T$;
in particular, 
any (contractive) approximate 
identity for $C^*(\T)$ is also a (contractive) 
approximate identity for $A\rtimes_\alpha\T$.
Recall that the dual automorphism
$\widehat{\alpha}\in\Aut(A\rtimes_\alpha\T)$ is 
given by 
$\widehat{\alpha}(fa)(z)=z f(z)a$ for all 
$z\in\T$. 
\end{rem}

Next, we show that approximate 
representability is dual to the
Rokhlin property.

\begin{prop}\label{thm:DualityRpAr} 
Let $A$ be a \uca, let $\alpha\colon\T\to\Aut(A)$ be an action, and let $\beta\in \Aut(A)$ be an automorphism.
\be
\item The action $\alpha$ has the \Rp\ if and only if $\widehat{\alpha}\in\Aut(A\rtimes_\alpha \T)$ is approximately representable.
\item The automorphism $\beta$ is approximately representable if and only if $\widehat{\beta}\colon \T\to\Aut(A\rtimes_\beta\Z)$ has the
Rokhlin property.
\ee
\end{prop}
\begin{proof} 
(1). Assume that $\alpha$ has the \Rp.
Let $F\subseteq A\rtimes_\alpha\T$ be a finite subset and let 
$\ep>0$. For $\xi\in L^1(\T)$ and $a\in A$, write
$\xi a$ for the 
for the function 
given by $(\xi a)(z)=\xi(z)a$ for all $z\in\T$.
Since the linear span of the 
elements of this form is dense in $L^1(\T,A)$, and
hence also in $A\rtimes_\alpha\T$,
we may assume that 
there exist finite subsets $F_A\subseteq A$ and $F_\T\subseteq L^1(\T)$ such that every element of $F$ has the form 
$\xi a$ for $a\in F_A$ and $\xi\in F_\T$.
Without loss of generality, we may assume that
the sets $F_A$ and $F_\T$ contain only self-adjoint contractions.

Let $f\colon \T\to\C$ be a positive, continuous 
function whose support is a small enough neighborhood of
$1\in \T$ so that the following conditions are satisfied:
\be\item[(i)] $\|(f\ast f)b-b\|<\ep$ for all 
$b\in F\cup \widehat{\alpha}(F)$;
\item[(ii)] $\|f\|_1=1=\|f\ast f\|_1$;
\item[(iii)] $f(z)=f_t(\overline{z})$ for all 
$z\in\T$;
\item[(iv)] with $\widetilde{f}(z)=z f(z)$ for all
$z\in\T$, we have 
\[\|f-\widetilde{f}\|_1<\ep \ \ \mbox{ and } \ \ 
\|f\ast f-\widetilde{f}\ast \widetilde{f}\|_1<\ep;
\]
\item[(v)] given $\xi\in F_\T$ and $a\in F_A$,
if $z,\sigma,\omega\in \T$ satisfy
$f(\omega)f(\overline{z}\sigma\omega)\neq 0$, then
\[\|\alpha_\omega(a)-a\|<\frac{\ep}{2} \ \ \mbox{ 
 and } \ \ |\xi(\sigma)-\xi(z)|<\frac{\ep}{2}.
\]
\ee 


Using \autoref{prop:EqSemiProj}, find a unitary $u\in A$
satisfying 
\be\item[(vi)] $\alpha_\zeta(u)=\zeta u$ for all $\zeta\in\T$, and 
\item[(vii)] $\|ua-au\|<\ep/2$ for all $a\in \bigcup_{\omega\in\T}\alpha_\omega(F_A)$.
\ee
We
regard $u$ as a unitary in the multiplier algebra of $A\rtimes_\alpha\T$ 
via the canonical unital embedding $A\hookrightarrow M(A\rtimes_\alpha\T)$, 
and set $v=fu^*$. Then $v$ is a contraction 
in $L^1(\T,A)$, and hence also in $A\rtimes_\alpha\T$. 
We proceed to check the conditions in 
\autoref{df:AppRep}. Let $z\in \T$. Then 
\begin{align*}
(v^*\ast v)(z)
&= \int_\T v^*(\omega)\alpha_\omega(v(\overline{\omega}z))\ d\omega
= \int_\T \omega f(\omega)u\alpha_\omega(f(\overline{\omega}z)u^*)\ d\omega\\
&\stackrel{\mathrm{(vi)}}{=} \int_\T \omega f(\omega)uf(\overline{\omega}z)\overline{\omega}u^*\ d\omega = (f\ast f)(z).
\end{align*}
Thus, $v^*\ast v=f\ast f$ and hence 
condition (a) in \autoref{df:AppRep} follows
from condition (ii) above. 
In order to check (b), we compute as follows for $z\in\T$:
\begin{align*}
(v\ast v^*)(z)
&= \int_\T v(\omega)\alpha_\omega(v^*(\overline{\omega}z))\ d\omega
= \int_\T f(\omega)u^*\alpha_\omega(\overline{\omega}z f(\overline{\omega}z)u)\ d\omega\\
&\stackrel{\mathrm{(vi)}}{=} \int_\T f(\omega)u^*\overline{\omega}z f(\overline{\omega}z)\omega u\ d\omega
= \int_\T \omega f(\omega)\overline{\omega}z f(\overline{\omega}z)\ 
d\omega
= (\widetilde{f}\ast \widetilde{f})(z).
\end{align*}
Thus $v\ast v^*=\widetilde{f}\ast \widetilde{f}$.
We deduce that 
\[\|v^*v-vv^*\|\leq \|v^*\ast v- v\ast v^*\|_1=\|f\ast f-
\widetilde{f}\ast \widetilde{f}\|_1\stackrel{\mathrm{(iv)}}{<}\ep,\]
as desired. In order to check (c), 
let $\zeta\in\T$. Then
\[\widehat{\alpha}(v)(z)=z f(z) u^*=\widetilde{f}(z)u^*,\]
so $\widehat{\alpha}(v)=\widetilde{f}u^*$.
Using this at the second step, we get
\[
\|\widehat{\alpha}(v)-v\|\leq \|\widehat{\alpha}(v)-v\|_1
=\|f-\widetilde{f}\|_1\stackrel{\mathrm{(iv)}}{<}\ep,\]
as desired.
Finally, in order to check (d), it suffices
to take $\xi\in F_\T$ and $a\in F_A$, and show the 
desired inequality for $b=\xi a$. 
Given
$z\in\T$, we have
\begin{align*}
(v\ast b \ast v^*)(z)
&=\int_\T v(\omega)\alpha_\omega((b \ast v^*)(\overline{\omega}z)) \ d\omega \\
&=\int_\T f(\omega)u^*\alpha_\omega\Big( 
\int_\T 
\xi(\sigma)a \alpha_\sigma(v^*(\overline{\sigma\omega}z)) 
\ d\sigma \Big) d\omega \\
&=\int_\T\int_\T f(\omega)u^*\xi(\sigma)\alpha_\omega(a)
\alpha_{\omega\sigma 
}\big(\overline{\sigma\omega}z
f(\overline{\sigma\omega}z)u
\big) d\sigma d\omega \\
&\stackrel{\mathrm{(vi)}}{=}
\int_\T\int_\T f(\omega)u^* 
\xi(\sigma)\alpha_\omega(a) 
\overline{\sigma\omega}z
f(\overline{\sigma\omega}z)
\omega\sigma u
\ d\sigma d\omega \\
&=
z\int_\T\int_\T  
\xi(\sigma)u^*\alpha_\omega(a)u f(\omega)
f(\overline{\sigma\omega}z)
\ d\sigma d\omega\\
&\stackrel{\mathrm{(vii)}}{\approx}_{\!\!\frac{\ep}{2}}
z\int_\T\int_\T  
\xi(\sigma)\alpha_\omega(a) f(\omega)
f(\overline{\sigma\omega}z)
\ d\sigma d\omega. 
\end{align*}
By (iv), if in the above expression
we replace $\xi(\sigma)\alpha_\omega(a)$
by $\xi(z)a$, we obtain an element in $A$ 
whose distance to $(v\ast b \ast v^*)(z)$ is at most 
$\ep/2$. 
Hence,
\begin{align*}
(v\ast b \ast v^*)(z)&\approx_{\ep} 
z \xi(z)a\int_\T\int_\T f(\omega) 
f(\overline{\sigma\omega}z)
\ d\sigma d\omega \\
&= z b(z)
\int_\T (f\ast f)(\overline{\sigma}z)\ d\sigma
= z b(z)\|f\ast f\|_1\stackrel{\mathrm{(ii)}}{=}
\widehat{\alpha}(b)(z).\end{align*}
We conclude that
\[
\|vbv^*-\widehat{\alpha}(b)\|\leq
\|v\ast b \ast v^*-\widehat{\alpha}(b)\|_1<\ep,
\]
as desired. This shows 
that $\widehat{\alpha}$ is approximately representable.

Conversely, assume that $\widehat{\alpha}$ is approximately representable. 
Denote the left regular representation of $G$ by 
$\lambda\colon \T\to\U(L^2(\T))$.
By Takai duality, there is a canonical equivariant identification
\[\label{eqn:3.1}\tag{2.3}(A\rtimes_\alpha\T\rtimes_{\widehat{\alpha}}\Z,\widehat{\widehat{\alpha}})\cong (A\otimes\K(L^2(\T)),\alpha\otimes\Ad(\lambda)).\] 

Let
$p\in \K(L^2(\T))$ be the projection onto the constant functions, and let $e\in M(A\rtimes_\alpha\T\rtimes_{\widehat{\alpha}}\Z)$
be the projection corresponding to $1_A\otimes p$
under the identification in (\ref{eqn:3.1}). Then $e$ and 
$p$ are $\T$-invariant, and 
there is a canonical equivariant isomorphism
\[\label{eqn:3.2}\tag{2.4}
\left(e(A\rtimes_\alpha\T\rtimes_{\widehat{\alpha}}\Z)e,\widehat{\widehat{\alpha}}\right)\cong (A,\alpha).\]

Let $u\in M(A\rtimes_\alpha\T\rtimes_{\widehat{\alpha}}\Z)$ be the canonical unitary implementing $\widehat{\alpha}$.
Let $F\subseteq A$ be a finite subset and let $\ep>0$.
Let $\delta>0$ such that whenever $x\in A$ satisfies $\|x^*x-1\|<\delta$
and $\|xx^*-1\|<\delta$, then there is $w\in \U(A)$ with $\|w-x\|<\ep/2$.
Set 
\[F''=\{e\}\cup \{a\otimes p\colon a\in F\}\subseteq A\rtimes_\alpha\T\rtimes_{\widehat{\alpha}}\Z.\]
Let $F'\subseteq A\rtimes_\alpha\T$ be a finite subset and let $n\in\N$ 
such that any
element in $F''$ is within $\ep/2$ of the span of $\{bu^k\colon b\in F', -n\leq k\leq n\}$.
Using approximate representability of $\widehat{\alpha}$, let $v\in A\rtimes_\alpha\T$
be a contraction satisfying conditions (a), (b), (c)
and (d) in \autoref{df:AppRep} for 
$\ep_0=\min\big\{\frac{\ep}{26n^2|F'|},\delta\big\}$ and $F'$. 
Set $y=v^*u$, which is a contraction in $A\rtimes_\alpha\T\rtimes_{\widehat{\alpha}}\Z$. Then 
\[\label{eqn:3.3}\tag{2.5}
\widehat{\widehat{\alpha}}_z(y)=v^*\widehat{\widehat{\alpha}}_z(u)=z v^*u=z y\]
for all $z\in\T$. Moreover, given $b\in F'$ and $k\in \Z$ with $|k|\leq n$, 
we have
\begin{align*}
ybu^k&=v^*ubu^k=v^*\widehat{\alpha}(b)u^{k+1}\stackrel{\mathrm{(b)}}{\approx}_{\!\ep_0}
v^*\widehat{\alpha}(b)v^*vu^{k+1}\\
&\stackrel{\mathrm{(a)}}{\approx}_{\!\ep_0}
v^*\widehat{\alpha}(b)vv^*u^{k+1}
\stackrel{\mathrm{(d)}}{\approx}_{\!\ep_0}
v^*vbv^*vv^*u^{k+1}\\
&\stackrel{\mathrm{(b)}}{\approx}_{\!2\ep_0}
bv^*u^{k+1}
\stackrel{\mathrm{(c)}}{\approx}_{\!k\ep_0}bu^kv^*u=bu^ky.
\end{align*}
It follows from the choice of $F'$, $n$ and $\ep_0$ that 
\[\label{eqn:3.4}\tag{2.6}\|yc-cy\|<\frac{\ep}{2}\] 
for every $c\in F''$.
Set $x=eye$, which we regard as an element in $A$. By
(\ref{eqn:3.3}) we have
$\alpha_z(x)=z x$ for all $z\in\T$, since $e$ is $\T$-invariant.
For $a\in F$ we have
\[\|xax^*-a\|=\|eye(a\otimes p)ey^*e-(a\otimes p)\|\leq \|y(a\otimes p)y^*-(a\otimes p)\|
\stackrel{(\ref{eqn:3.4})}{<}
\frac{\ep}{2},\]
since $a\otimes p$ belongs to $F''$.
Moreover, $\|x^*x-1\|=\|ey^*eye-e\|<\ep_0\leq \delta$, and similarly $\|xx^*-1\|<\delta$.
By the choice of $\delta$, there exists a unitary $w\in A$ such that $\|w-x\|<\ep/2$.
It is then straightforward to check that $\|wa-aw\|<\ep$ for all $a\in F$ and 
$\max_{z\in\T}\|\alpha_z(w)-z w\|<\ep/2$.  
This shows that $\alpha$ has the \Rp.

(2). Assume that $\beta$ is approximately representable. Let $F$ be a finite subset of $A\rtimes_{\beta}\Z$ and let $\varepsilon>0$. Denote by $u$ the canonical unitary in the
crossed product. Since $A$ and $u$ generate $A\rtimes_{\beta}\Z$, one can assume that $F=F'\cup\{u\}$, where $F'$ is a finite subset of $A$.
Using approximate representability for $\beta$, find $v\in \U(A)$ with
\bi\item $\|\beta(b)-vbv^*\|<\varepsilon$ for all $b\in \beta(F')$, and
\item $\|\beta(v)-v\|<\varepsilon$.\ei

Set $w=v^*u$, which is a unitary in $A\rtimes_\beta\Z$. For $b\in F'$ we have
\[wb=v^*ub=v^*\beta(b)u\approx_\ep bv^*u=bw.\]
Moreover, 
\[wu=v^*uu=u\beta^{-1}(v)u\approx_\ep uw.\]
It follows that $\|wa-aw\|<\ep$ for all $a\in F=F'\cup\{u\}$.
On the other hand,
\[\widehat{\beta}_ z (w)=\widehat{\beta}_ z (v^*u)=v^*( z u)= z w\]
for all $ z\in \T$. Thus, $w$ is the desired unitary, and $\widehat{\beta}$ has the Rokhlin property.

Conversely, assume that $\widehat{\beta}$ has the Rokhlin property.
We continue to denote by $u$ the canonical unitary in $A\rtimes_\beta\Z$ that implements $\beta$.
Let $F\subseteq A$ be a finite subset and set $F'=F\cup\{u\}\subseteq A\rtimes_\beta\Z$.
Use \autoref{prop:EqSemiProj}
to choose a unitary $w\in \U(A\rtimes_\beta\Z)$ such that
\bi\item $\widehat{\beta}_ z(w)= z w$ for all $ z\in \T$;
\item $\|wb-bw\|=0$ for all $b\in F'$.\ei

Set $v=uw^*\in A\rtimes_\beta\Z$. Then the first
condition above implies that
$\widehat{\beta}_ z(v)=v$ for all $ z\in\T$, and hence the unitary $v$ belongs to the fixed point algebra
$(A\rtimes_\beta\Z)^{\widehat{\beta}}$, which equals $A$
by Proposition~7.8.9 in~\cite{Ped_algebras_1979}. 
For $a\in F$, we have
\[
\|vav^*-\beta(a)\|=\|uw^*awu^*-uau^*\|=\|w^*aw-a\|<\ep.
\]
Moreover, 
$\|\beta(v)-v\|=\|u(uw^*)u^*-uw^*\|<\ep$, since $u\in F'$.
It follows that $v$ satisfies the conditions of \autoref{df:AppRep},
so $\beta$ is approximately representable.\end{proof}

Using \autoref{thm:DualityRpAr}, we show that for actions with 
the \Rp, every ideal in the crossed product is induced by an 
invariant ideal of the algebra.

\begin{prop}\label{simplicity is preserved} Let $A$ be a \uca\ and let $\alpha\colon\T\to\Aut(A)$ have the \Rp. Then 
every ideal in $A^\alpha$ has the form $I\cap A^\alpha$, for some 
$\T$-invariant ideal $I$ in $A$, and every ideal in $A\rtimes_\alpha\T$ has the form $I\rtimes_\alpha \T$ for some $\T$-invariant ideal 
$I$ in $A$. 
In particular, if $A$ is simple then so are $A^\alpha$ and 
$A\rtimes_\alpha \T$.
\end{prop}
\begin{proof} 
Since $A^\alpha\otimes\K(\ell^2(\Z))\cong A\rtimes_\alpha\T$ 
by \autoref{cor: FixedPtCP}, it is enough to show the statement for 
$A^\alpha$. We identify $A$ with $A^\alpha\rtimes_{\check{\alpha}}\Z$,
where $\check{\alpha}$ is the predual of $\alpha$; see \autoref{df:predual}. Let $J$ be an ideal in $A^\alpha$. Since $\check{\alpha}$ is
approximately inner by part~(2) of \autoref{thm:DualityRpAr}, it follows that $\check{\alpha}(J)=J$. Hence $I=J\rtimes_{\check{\alpha}}\Z$ is 
canonically an ideal in $A$ satisfying $I\cap A^\alpha=J$, as desired.\end{proof}

\section{\texorpdfstring{$K$}{K}-theoretic obstructions to the Rokhlin property}

In this section, we study the $K$-theory of \ca s that admit circle
actions with the \Rp. First, we show that 
the canonical
inclusion $A^\alpha\to A$ induces an order-embedding
$K_\ast(\iota)\colon K_\ast(A^\alpha)\to K_\ast(A)$; see \autoref{prop:InjKthy}. Moreover, the quotient of $K_\ast(A)$
by $K_\ast(A^\alpha)$ can be canonically identified with 
$K_\ast(SA^\alpha)$, and the induced extension
\[\xymatrix{0\ar[r] & K_\ast(A^\alpha)\ar[r] &K_\ast(A)\ar[r] & K_\ast(SA^\alpha)\ar[r]& 0}\]
is pure; see \autoref{df:PureExt} and \autoref{thm:PureExt}.

In the following proposition, we will use the fact that if
$\alpha\colon \T\to\Aut(A)$ has the \Rp, then so does
its $n$-amplification $\alpha\otimes\id_{M_n}\colon \T \to\Aut(M_n(A))$ for every
$n\in\N$, and that $M_n(A)^{\alpha\otimes \id_{M_n}}=M_n(A^\alpha)$.

\begin{prop}\label{prop:InjKthy}
Let $A$ be a unital \ca\ and let $\alpha\colon \T\to\Aut(A)$ be an action with the \Rp, and let
$\iota\colon A^\alpha\hookrightarrow A$ denote the canonical inclusion. Then $K_0(\iota)(x)\leq K_0(\iota)(y)$ in $K_0(A)$ if and only if $x\leq y$ in $K_0(A^\alpha)$. In particular, $K_0(\iota)$
and $K_1(\iota)$ are injective.
\end{prop}
\begin{proof}
%
Let 
$x,y\in K_0(A^\alpha)$, and assume that 
$K_0(\iota)(x)\leq K_0(\iota)(y)$. Set 
$z=y-x$, so that $K_0(\iota)(z)\geq 0$ in $K_0(A)$. Find projections 
$p,q\in \bigcup_{m\in\N}M_m(A^\alpha)$ with $z=[p]-[q]$ in $K_0(A^\alpha)$, and let $e\in \bigcup_{m\in\N}M_m(A)$ satisfy $K_0(\iota)(z)=[e]$ in $K_0(A)$. Then $[p]=[e]+[q]$ in $K_0(A)$.
By increasing the matrix sizes, may assume that there 
exists $n\in\N$ such that:
\bi\item $p$ and $q$ belong to
$M_n(A^\alpha)$;
\item $e$ belongs to $M_n(A)$;
\item $e$ is orthogonal to $q$;
\item $p$ is Murray-von Neumann equivalent to $e+q$ in $M_n(A)$.\ei

Set $\alpha^{(n)}=\alpha\otimes\id_{M_n}$, which
has the Rokhlin property. 
Let $s\in M_n(A)$ be a partial isometry satisfying $s^*s=p$ and 
$ss^*=e+q$.
For $\ep=1/12$, let $\delta>0$ such that whenever $B$ is a \ca\ and $a\in B_{\mathrm{sa}}$ satisfies $\|a^2-a\|<\delta$, then there exists a projection $r\in B$
with $\|r-a\|<\ep$.
Set $\ep_0=\min\{\ep,\delta/5\}$, and 
let $\sigma\colon M_n(A)\to M_n(A^\alpha)$ be a unital completely 
positive map as in the conclusion of Theorem~2.11 in~\cite{Gar_CptRok}
for $\ep_0$, $F_1=\{p,q,e,s,s^*\}$
and $F_2=\{p,q\}$. Set $t=\sigma(s)$ and $f=\sigma(e)$. 
Then 
\[t^*t\approx_{\ep_0} \sigma(s^*s)=\sigma(p)\approx_{\ep_0} p,\]
so $\|t^*t-p\|<2\ep_0$. In particular,
$(1-q)f\approx_{2\ep_0} f$. Moreover, $f^*=f$ and 
\[f^2=\sigma(e)^2\approx_{2\ep_0}\sigma(e^2)=\sigma(e)=f,\]
thus $\|f^2-f\|<2\ep_0$.
On the other hand, 
\[qf=q\sigma(e)\approx_{\ep_0} \sigma(q)\sigma(e)\approx_{\ep_0} \sigma(qe)=0,\]
and hence $\|qf\|<2\ep_0$. Similarly, $\|fq\|<2\ep_0$.

Set $a=(1-q)f(1-q)$. Then $a=a^*$ and $\|a-f\|\leq 4\ep_0$. Moreover,
\[a^2=(1-q)f(1-q)f(1-q)\approx_{2\ep_0}(1-q)f^2(1-q)\approx_{2\ep_0}(1-q)f(1-q)=a,\]
so $\|a^2-a\|\leq 4\ep_0<\delta$. 
By the choice of 
$\delta$ applied to $B=(1-q)M_n(A^\alpha)(1-q)$, 
there exists a projection $r\in M_n(A^\alpha)$ with $\|r-a\|<\ep$
and $rq=0$. 

Set $x=(q+r)tp\in M_n(A^\alpha)$. Then 
\begin{align*}
x^*x&=pt^*(q+r)tp\approx_\ep pt^*(q+a)tp\\
&\approx_{4\ep_0}pt^*(q+f)tp \approx_{2\ep_0}pt^*tt^*tp\\
&\approx_{2\ep_0}p^4=p,
\end{align*}
so $\|x^*x-p\|\leq 8\ep_0+\ep<1$. Similarly, we have
\begin{align*}
xx^*&=(q+r)tpt^*(q+r)\approx_{2\ep_0} (q+r)tt^*tt^*(q+r)\\
&\approx_{4\ep_0}(q+r)(q+f)(q+f)(q+r) \approx_{4\ep_0+\ep}(q+r)^4=q+r,
\end{align*}
so $\|xx^*-(q+r)\|\leq 10\ep_0+\ep<1$. Since $x=(q+r)xp$, it follows 
from Lemma~2.5.3 in \cite{Lin_Book} 
that there exists a partial isometry $v\in M_n(A^\alpha)$ 
such that $v^*v=p$ and $ww^*=q+r$. 
It follows that 
$[p]=[q]+[r]$ in $K_0(A^\alpha)$, and thus $z=[p]-[q]$ is 
positive in $K_0(A^\alpha)$, as desired.

It follows that $K_0(\iota)$ is injective. The 
result for $K_1$ follows by replacing $A$ with its tensor product with any unital nuclear \ca\ $B$ 
with $B\sim_{KK}C_0(\mathbb{R})$,
endowed with the trivial action.
\end{proof}

Next, we recall the definition of a pure subgroup and a pure extension.

\begin{df}\label{df:PureExt}
Let $G$ be an abelian group and let $H$ be a subgroup. We say that $H$ is \emph{pure} if $nH=nG\cap H$ for all $n\in \N$. 
In other words, for every $h\in H$ and $n\in\N$, if there exists
$g\in G$ with $ng=h$ then there exists $h'\in H$ with $nh'=h$.

An extension $0\to H\to G\to Q\to 0$ is said to be \emph{pure} 
if $H$ is pure in $G$.\end{df}

The notion of a pure subgroup generalizes that of a direct summand. 
For example, the torsion subgroup of
any abelian group is always a pure subgroup, although it is not 
always a direct summand. If $G$ is finitely generated, then any 
pure subgroup is automatically a direct summand. On the other hand, 
there exist pure subgroups which are finitely generated, yet not a direct summand (despite being direct summands in every finitely
generated subgroup that contains them). 

In the following theorem, note that part (3) does not
follow from part~(2) since in (3) we only make assumptions
about \emph{one} of the $K$-groups of $A$, and not
both.

\begin{thm}\label{thm:PureExt} 
Let $\alpha\colon \T\to\Aut(A)$ be an action on a unital \ca\ $A$ with the \Rp. Denote by 
$\iota\colon A^\alpha\to A$ the canonical inclusion.
\be\item
There is a canonical class 
$\Ext_\ast(\alpha)=(\Ext_0(\alpha),\Ext_1(\alpha))$,
where $\Ext_j(\alpha)\in \Ext(K_j(SA^\alpha),K_j(A^\alpha))$ is the pure extension
\[\xymatrix{0\ar[r] & K_j(A^\alpha)\ar[r]^-{K_j(\iota)} &K_j(A)\ar[r] & K_j(SA^\alpha)\ar[r]& 0}.\]
\item If both $K_0(A)$ and $K_1(A)$ are (possibly infinite) direct sums 
of cyclic groups, then there are isomorphisms 
\[K_0(A)\cong K_1(A)\cong K_0(A^\alpha)\oplus K_1(A^\alpha)\] 
such that $[1_A]\in K_0(A)$ is sent to
$([1_{A^\alpha}],0)\in K_0(A^\alpha)\oplus K_1(A^\alpha)$.
\item If at least one of $K_0(A)$ or $K_1(A)$ is finitely generated, then 
there are isomorphisms as in~(2).
\ee
\end{thm}
\begin{proof} 
(1). 
Consider the Pimsner-Voiculescu exact sequence associated 
to $\check{\alpha}$:
\beqa \xymatrix{ K_0(A^\alpha)\ar[rr]^-{1-K_0(\check{\alpha})}&& K_0(A^\alpha)\ar[rr]^{K_0(\iota)} && K_0(A)\ar[d]\\
K_1(A)\ar[u]&& K_1(A^\alpha)\ar[ll]^-{K_1(\iota)}&& K_1(A^\alpha)\ar[ll]^-{1-K_1(\check{\alpha})}.}\eeqa
By \autoref{prop:InjKthy}, $K_\ast(\iota)$ is injective
and thus we obtain the extension
\[0\to K_\ast(A^\alpha)\to K_\ast(A) \to K_{\ast}(SA^\alpha)\to 0.\]
We claim that the extension is pure. By taking suspensions,
it suffices
to prove that $K_0(A^\alpha)$ is a pure subgroup of $K_0(A)$. Let 
$x\in K_0(A^\alpha)$, let $k\in\N$ and let $y\in K_0(A)$, and 
suppose that $ky=K_0(\iota)(x)$. 
Find projections $p_x,q_x\in \bigcup_{m\in\N} M_m(A^\alpha)$ and 
$p_y,q_y\in \bigcup_{m\in\N} M_m(A)$ such that $x=[p_x]-[q_x]$ 
and $y=[p_y]-[q_y]$. It follows that 
$k[p_y]+[q_x]=[p_x]+k[q_y]$ in $K_0(A)$.
Without loss of generality, we may assume that there is $n\in\N$
such
that 
\bi\item $p_x,q_x$ to
$M_n(A)^\alpha$ and $p_y,q_y$ belong to $M_n(A)$;
\item $p_x$ is orthogonal to $q_y$ and $p_y$ is orthogonal to $q_x$;
\item there exists $s\in M_{nk}(A)$ with  
\begin{align*} s^*s=\left(
               \begin{array}{cccc}
                 q_x+p_y &  &  &  \\
                  & p_y &  &  \\
                  & & \ddots &  \\
                  & &  & p_y \\
               \end{array}
             \right) \ \ \mbox{ and } \ \ ss^*=\left(
               \begin{array}{cccc}
                 p_x+q_y &  &  &  \\
                  & q_y &  &  \\
                  & & \ddots &  \\
                  & &  & q_y \\
               \end{array}
             \right).\end{align*}
\ei

Note that $\alpha^{(nk)}\colon \T\to\Aut(M_{nk}(A))$ has the 
\Rp.
For $\ep=1/12$, let $\delta>0$ such that whenever $B$ is a \ca\ and $a\in B$ is a self-adjoint
element satisfying $\|a^2-a\|<\delta$, then there exists a projection $r\in B$
with $\|r-a\|<\ep$.
Set $\ep_0=\min\{\ep,\delta/10\}$, and 
let $\sigma\colon M_{nk}(A)\to M_{nk}(A^\alpha)$ be a unital completely 
positive map as in the conclusion of Theorem~2.11 in~\cite{Gar_CptRok}
for $\ep_0$, $F_1=\{p_x,q_x,p_y,q_y,s,s^*\}$
and $F_2=\{p_x,q_x\}$. 
Then $\|\sigma(p_x)-p_x\|<\ep$ and $\|\sigma(q_x)-q_x\|<\ep$. 

Set $e_y=\sigma(p_y)$ and $f_y=\sigma(q_y)$, which are self-adjoint
contractions in $M_{nk}(A^\alpha)$ satisfying $\|e_y^2-e_y\|<2\ep_0$ 
and $\|f_y^2-f_y\|<2\ep_0.$ Set $t=\sigma(s)$.
Then 
\[t^*t\approx_{2\ep_0}\diag(q_x+e_y,e_y,\ldots,e_y) \ \mbox{ and } 
\ tt^*\approx_{2\ep_0}\diag(p_x+f_y,f_y,\ldots,f_y).
\]

Set $a=(1-p_x)f_y(1-p_x)$ and $b=(1-q_x)e_y(1-q_x)$, which are self-adjoint elements
in the corners of $M_{nk}(A^\alpha)$ by $1-p_x$ and $1-q_x$, respectively. 
Moreover, $\|a-f_y\|\leq 4\ep_0$ and $\|b-e_y\|\leq 4\ep_0$.
On the other hand, 
\begin{align*}a^2&=(1-p_x)f_y(1-p_x)f_y(1-p_x)\\
&\approx_{2\ep_0}(1-p_x)f_y^2(1-p_x)\\
&\approx_{2\ep_0}(1-p_x)f_y(1-p_x)=a,\end{align*}
so $\|a^2-a\|<4\ep_0<\delta$. Similarly, we have $\|b^2-b\|<\delta$. 
Using the definition of $\delta$ with $B=(1-p_x)M_{nk}(A^\alpha)(1-p_x)$, there
exits a projection $\widetilde{q}_y \in M_{nk}(A^\alpha)$ satisfying
$\|\widetilde{q}_y-a\|<\ep$ and $p_x\widetilde{q}_y=0$. Similarly,
there
exits a projection $\widetilde{p}_y \in M_{nk}(A^\alpha)$ satisfying
$\|\widetilde{p}_y-b\|<\ep$ and $\widetilde{p}_yq_x=0$.
Set 
\[r=\diag(p_x+\widetilde{q}_y,\widetilde{q}_y,\ldots,\widetilde{q}_y) t
\diag(q_x+\widetilde{p}_y,\widetilde{p}_y,\ldots,\widetilde{p}_y),\]
which belongs to $M_{nk}(A^\alpha)$. One checks
that $\|r^*r-\diag(q_x+\widetilde{p}_y,\widetilde{p}_y,\ldots,\widetilde{p}_y)\|<1$, 
and that 
$\|r^*r-\diag(p_x+\widetilde{q}_y,\widetilde{q}_y,\ldots,\widetilde{q}_y)\|<1$. By
Lemma~2.5.3 in \cite{Lin_Book}, there exists a partial isometry $w\in M_{nk}(A^\alpha)$ such that 
\[w^*w=\left(
               \begin{array}{cccc}
                 q_x+\widetilde{p}_y &  &  &  \\
                  & \widetilde{p}_y &  &  \\
                  & & \ddots &  \\
                  & &  & \widetilde{p}_y \\
               \end{array}
             \right) \ \mbox{ and } \ ww^*=
\left(
               \begin{array}{cccc}
                 p_y+\widetilde{q}_y &  &  &  \\
                  & \widetilde{q}_y &  &  \\
                  & & \ddots &  \\
                  & &  & \widetilde{q}_y \\
               \end{array}
             \right).\] 
It follows that 
$[p_x]+k[\widetilde{q}_y]=k[\widetilde{p}_y]+[p_y]$ in $K_0(A^\alpha)$. With
$z=[\widetilde{q}_x]-[\widetilde{q}_y]\in K_0(A^\alpha)$, we have $kz=x$, as 
desired. 

(2). We show that there is an isomorphism $K_0(A)\cong K_0(A^\alpha)\oplus K_1(A^\alpha)$; the argument for $K_1(A)$ is identical because the assumptions are symmetric. 
Abbreviate $K_j(A)$ to $K_j$, and $K_j(A^\alpha)$ to $K_j^\alpha$, for $j=0,1$.
It is a standard result in group theory, usually attributed to
Kulikov, that subgroups of direct sums of cyclic groups are again 
direct sums of cyclic groups; see Theorem~18.1 in~\cite{Fuc_direct_1970}.
Since $K_1$ is a direct sum of cyclic groups, we deduce that the same is true for $K_1^\alpha$. 

By part~(1) of this theorem, there is a canonical quotient map 
$\pi\colon K_0\to K_1^\alpha$ whose kernel is $K_0^\alpha$. Choose 
a presentation $K_1^\alpha\cong \bigoplus_{s\in S}C_s$, where each $C_s$ is a cyclic group with generator $x_s$. In particular, 
$\{x_s\colon s\in S\}$ generates $K_1^\alpha$.
Let $s\in S$. If $x_s$ has infinite order in $K_1^\alpha$, we let 
$y_s\in K_0$ be any group element (necessarily of infinite order) 
satisfying $\pi(y_s)=x_s$. If $x_s$ has order $n<\I$, let $z_s\in K_0$ be 
any lift of $x_s$, and note that $nz_s$ belongs to $K_0^\alpha$, which is 
a pure subgroup of $K_0$. Hence there exists $k_s\in K_0^\alpha$
with $nk_s=nz_s$, and we set $y_s=z_s-k_s$, which also lifts $x_s$. 

Let $L$ be the subgroup of $K_0$ generated by $\{y_s\colon s\in S\}$, which 
is mapped isomorphically onto $K_1^\alpha$ via $\pi$. In particular, 
$K_0^\alpha\cap L=\{0\}$ and $K_0^\alpha+L=K_0$. (Equivalently, $L$
defines a splitting for the quotient map $\pi$.) We deduce that the extension
$0\to K_0^\alpha \to K_0\to K_1^\alpha\to 0$ splits, and thus 
$K_0\cong K_0^\alpha\oplus K_1^\alpha$. The isomorphism can be clearly chosen
to send $[1_A]\in K_0$ to $[1_{A^\alpha}]\in K_0^\alpha$. 

(3).
Assume that $K_1(A)$ 
is finitely generated (and in particular a direct sum of cyclic groups). 
Hence $K_0(A^\alpha)$ and $K_1(A^\alpha)$ are both finitely generated, being a quotient and a subgroup of $K_1(A)$, respectively. 
The argument given in the proof of part~(2) above 
shows that there is an isomorphism 
$K_0(A)\cong K_0(A^\alpha)\oplus K_1(A^\alpha)$. Thus $K_0(A)$ is 
also finitely generated, and repeating the same argument again, exchanging 
the roles of $K_0$ and $K_1$, shows that 
$K_1(A)\cong K_0(A^\alpha)\oplus K_1(A^\alpha)$. 
\end{proof}

In reference to part~(2) of \autoref{thm:PureExt}, we mention that it is 
not in general true that a pure subgroup of a direct sum of cyclic groups 
is automatically a direct summand. For example, set $G=\bigoplus_{n\in\N} \Z_{2^n}$, with canonical generators $x_n\in\Z_{2^n}$ for $n\in\N$, and 
let $H$ be the subgroup generated by $\{x_n-2x_{n+1}\colon n\in\N\}$. Then 
$H$ is pure in $G$ but not a direct summand.

\section{Circle actions on Kirchberg algebras}
This section contains our main results concerning
$KK^\T$-theory for Rokhlin actions. This 
includes Theorems~\ref{thmintro:Classif}, 
\ref{thmintro:KKeqKir} and \ref{thmintro:Emb}
from the introduction. 
In the presence of the UCT, any isomorphism between
the pure extensions from \autoref{thm:PureExt}
lifts to a $KK^\T$-equivalence. 
We show by means of an 
example that an isomorphism of the $K$- and 
$K^\T$-theories
does not necessarily lift to a $KK^\T$-equivalence;
see \autoref{eg:EqKThyNotEnough}
Finally, we also describe the extensions that arise
as $\Ext_\ast(\alpha)$ for a Rokhlin action $\alpha$ on
a Kirchberg algebra satisfying the UCT; see 
\autoref{thm:rangeInv}.


\begin{df}
Let $A$ be a simple \uca. Then $A$ is said to be:
\be
\item \emph{purely infinite}, if for every $a\in A\setminus\{0\}$ there are $x,y\in A$ with $xay=1$.
\item a \emph{Kirchberg algebra}, if it is purely infinite, separable and nuclear.
\ee
\end{df}

Recall that an automorphism $\varphi$ of a \ca\ is said to be 
\emph{aperiodic} if $\varphi^n$ is not inner for all $n\geq 1$.

\begin{prop} \label{prop:KirAperiodic} 
Let $A$ be a \uca, and let $\alpha\colon \T\to\Aut(A)$ be an 
action with the \Rp.
\be\item $A$ is simple if and only if $A^\alpha$ is simple and
$\check{\alpha}$ is aperiodic.
\item $A$ is purely infinite simple if and only if $A^\alpha$ is purely
infinite simple and $\check{\alpha}$ is aperiodic.
\item $A$ is a Kirchberg algebra if and only if $A^\alpha$ is a 
Kirchberg algebra and $\check{\alpha}$ is aperiodic. 
\item $A$ satisfies the UCT if and only if $A^\alpha$ satisfies the UCT.
\ee \end{prop}
\begin{proof} 
(1). If $A$ is simple, then $A^\alpha$ is simple by \autoref{cor: FixedPtCP}.
We show that $\check{\alpha}$ is aperiodic. 
Arguing by contradiction, suppose that there exist $n\geq 1$ and a unitary $v\in A^\alpha$ such that $\check{\alpha}^n=\Ad(v)$.
Set
\[w=v\check{\alpha}(v)\cdots\check{\alpha}^{n-1}(v).\]
Then $\check{\alpha}^{n^2}=\Ad(w)$.
Using that $\check{\alpha}^n(v)=v$ at the second step, that $xv=\check{\alpha}^n(x)v$ for all $x\in A^\alpha$ at the third, and that $\check{\alpha}^{-n}(v)=v$ at he fifth, we get 
\begin{align*}
\check{\alpha}(w)&=
 \check{\alpha}(v)\check{\alpha}^2(v)\cdots\check{\alpha}^{n-1}(v) \check{\alpha}^{n}(v)= \check{\alpha}(v)\check{\alpha}^2(v)\cdots\check{\alpha}^{n-1}(v) v\\
&=v\check{\alpha}^{-n}\big(\check{\alpha}(v)\check{\alpha}^2(v)\cdots\check{\alpha}^{n-1}(v) \big)=
v\check{\alpha}\big(\check{\alpha}^{-n}(v)\big)\check{\alpha}^2\big(\check{\alpha}^{-n}(v)\big)\cdots\check{\alpha}^{n-1}(\check{\alpha}^{-n}(v)\big)\\
&=v\check{\alpha}(v)\cdots\check{\alpha}^{n-1}(v)=w.
\end{align*}
It follows that $w$ is $\check{\alpha}$-invariant. 
With $u$ denoting the canonical unitary in $A^\alpha\rtimes_{\check{\alpha}}\Z$ that 
implements $\check{\alpha}$,
we therefore have $uwu^*=w$.
Set $z=u^{n^2}w^*$. 
It is clear that $z$ commutes with $u$, and for $a\in A^\alpha$ we have
$$zaz^*=u^{n^2}w^*aw\big(u^{n^2}\big)^*=u^{n^2}\check{\alpha}^{-n^2}(a)\big(u^{n^2}\big)^*=a,$$
so $z$ belongs to the center of $A^\alpha\rtimes_{\check{\alpha}}\Z\cong A$. 
Since $A$ is simple, its center is trivial and thus there is 
$\lambda\in\C$ with
$u^{n^2}=\lambda w$. In particular, $u^{n^2}$ belongs to $A^\alpha$, which
is a contradiction. This shows that $\check{\alpha}$ is aperiodic.

Conversely, if $\check{\alpha}$ is aperiodic and $A^\alpha$ is simple, 
it follows from Theorem 3.1 in \cite{Kis_outer} that the crossed product 
$A^\alpha\rtimes_\varphi\Z\cong A$ is simple.

(2). Assume that $A$ is purely infinite simple. Then $\check{\alpha}$
is aperiodic by part~(1).
Let $a\in A^\alpha$ be nonzero and let $\varepsilon>0$ small enough
so that $\ep^3+3\ep<1$. 
Without loss of generality, we assume that $\|a\|=1$.
Find $x,y\in A$ such that $xay=1$. 
By Lemma 4.1.7 in \cite{Ror_BookClassif},
we may assume that $\|x\|<1+\varepsilon$ and $\|y\|<1+\varepsilon$. 
Let $\sigma\colon A\to A^\alpha$ a completely positive unital map as 
in the conclusion of Theorem~2.11 in~\cite{Gar_CptRok}
for $\ep>0$, 
$F_2=\{x,y,xa,a\}$ and $F_1=\{a\}$. 
Then
\begin{align*} \sigma(x)a\sigma(y)&\approx_{(1+\ep)^2\ep}\sigma(x)\sigma(a)\sigma(y)
 \approx_{\ep} \sigma(xa)\sigma(y)
 \approx_{\ep} \sigma(xay)=1.
\end{align*}
Hence $\|\sigma(x)a\sigma(y)-1\|\leq \ep^3+3\ep<1$. It follows that $\sigma(x)a\sigma(y)$ is invertible. 
With $z\in A^\alpha$ denoting its inverse, we
have $\sigma(x)a\sigma(y)z=1$, as desired. 

Conversely, assume that $\check{\alpha}$ is aperiodic and $A^\alpha$ is purely infinite simple. Then $A\cong A^\alpha\rtimes_{\check{\alpha}}\Z$
is purely infinite simple by Corollary~4.6 in~\cite{JeoOsa}.

(3). This follows from (1) and (2), since $A$ is nuclear (respectively, separable)
if and only if so is $A^\alpha$.

(4). If $A$ satisfies the UCT, then so does $A^\alpha$ by Theorem~3.13 in~\cite{Gar_CptRok}. The converse follows from the 
fact that the UCT is preserved by $\Z$-crossed products.
\end{proof}

\begin{rem}
Let the notation be as in \autoref{prop:KirAperiodic}.
\bi\item Simplicity of $A$ is really needed in~(1) to conclude that $\check{\alpha}$ is aperiodic, even if $A^\alpha$ is simple. 
Consider for example the 
trivial automorphism of $\C$, whose dual action is $\texttt{Lt}$.
\item Outerness of $\alpha$ is not enough in~(2) to 
deduce pure infiniteness of 
$A^\alpha$ from pure infiniteness of $A$ (unlike for finite groups).
For example, the fixed point algebra of $\OI$ by its gauge action is 
AF (and not even simple).
\ei
\end{rem}

We will need some terminology.

\begin{df}\label{df:UniKKEq}
Let $G$ be a second countable, 
locally compact group (in this work either $\T$, $\Z$, or
the trivial group), let 
$A$ and $B$ be separable, unital \ca s, and let 
$\alpha\colon G\to\Aut(A)$ and $\beta\colon G\to\Aut(B)$
be actions.
We say that $(A,\alpha)$ and $(B,\beta)$ are 
\emph{unitally $KK^G$-equivalent},
if there is an invertible class $\eta\in KK^G(A,B)$
with $[1_A]\times \eta=[1_B]$.
In this situation, we say that $\eta$ is a
\emph{unital $KK^G$-equivalence}.
\end{df} 

A unital $KK^\Z$-equivalence between two $\Z$-actions
$(A,\sigma)$ and $(B,\theta)$ naturally induces an isomorphism
between the Pimsner-Voiculescu 6-term exact sequences of $\sigma$
and $\theta$ which moreover preserves the classes of the units. 
If the automorphisms are moreover approximately inner,
then each of these 6-term exact sequences splits into two short 
exact sequences. The resulting equivalence relation for sums of
short exact sequences (with distinguished classes) is the 
following:

\begin{df}\label{df:IsomExt}
For $j=0,1$, 
let $K^\mathcal{E}_j, G^\mathcal{E}_j, K^\mathcal{F}_j$, and $G^\mathcal{F}_j$ be countable abelian groups, let $k^\mathcal{E}_0\in K^\mathcal{E}_0$
and $k^\mathcal{F}_0\in K^\mathcal{F}_0$, 
and let 
\[(\mathcal{E}_j) \ \ 0 \to K^\mathcal{E}_j \to G^\mathcal{E}_j \to K^\mathcal{E}_{1-j}\to 0 
\ \ \ \ \mbox{and } \ \ \ \ 
(\mathcal{F}_j) \ \ 0 \to K^\mathcal{F}_j \to G^\mathcal{F}_j \to K^\mathcal{F}_{1-j}\to 0
\]
be short exact sequences. We say that $(\mathcal{E}_0,\mathcal{E}_1,k_0^\mathcal{E})$ is \emph{isomorphic} to $(\mathcal{F}_0,\mathcal{F}_1,k_0^\mathcal{F})$
if there exist group isomorphisms
$\varphi_j \colon K_j^\mathcal{E}\to K_j^{\mathcal{F}}$ and 
$\psi_j \colon G_j^\mathcal{E}\to G_j^\mathcal{F}$ with
$\varphi_0(k_0^\mathcal{E})=k_0^\mathcal{F}$
making 
the following diagram commute:
\begin{align*}\label{eqn:Commut}\tag{4.1}
\xymatrix{K_j^\mathcal{E} \ar[rr] \ar[d]^{\psi_j} && G_j^\mathcal{E}\ar[d]^{\varphi_j} \ar[rr] && K_{1-j}^\mathcal{E}\ar[d]^-{\psi_{1-j}}\\
K_j^\mathcal{F} \ar[rr]&& G_j^\mathcal{F}\ar[rr] && K_{1-j}^\mathcal{F}.}\end{align*}
\end{df}

This is different from having two isomorphisms $(\mathcal{E}_0,k_0^\mathcal{E})\cong(\mathcal{F}_0,k_0^\mathcal{F})$ and $\mathcal{E}_1\cong \mathcal{F}_1$, since, for example, we require the isomorphism $K_0^\mathcal{E}\cong K_0^\mathcal{F}$ to be the same both in the isomorphism
$\mathcal{E}_0\cong \mathcal{F}_0$ and in $\mathcal{E}_1\cong\mathcal{F}_1$.

We are now ready to 
prove that Rokhlin actions of the circle
on unital Kirchberg algebras are conjugate if they are 
unitally $KK^\T$-equivalent.

\begin{thm}\label{thm:ClassRpKir} 
Let $A$ and $B$ be unital Kirchberg algebras, and let
$\alpha\colon \T\to\Aut(A)$ and $\beta\colon \T\to\Aut(B)$ 
be actions
with the \Rp. 
Then $(A,\alpha)$ and $(B,\beta)$ are conjugate
if and only if they are unitally $KK^\T$-equivalent.

When $A$ and $B$ satisfy the UCT, these conditions 
are equivalent to the existence of an 
isomorphism $(\Ext_\ast(\alpha),[1_{A^\alpha}])\cong
(\Ext_\ast(\beta),[1_{B^\beta}])$ in the sense of \autoref{df:IsomExt}.
\end{thm}
\begin{proof}
Assume that $\alpha$ and $\beta$ are unitally
$KK^\T$-equivalent, and fix a unital 
$KK^\T$-equivalence
$\xi\in KK^\T\big((A,\alpha),(B,\beta)\big)$.
Denote by $\xi\rtimes\T$ the $KK^\Z$-equivalence
between $(A\rtimes_\alpha\T,\widehat{\alpha})$ 
and $(B\rtimes_\beta\T,\widehat{\beta})$ that 
$\xi$ induces. Combining Takai duality with 
\autoref{thm:RpisDual}, it follows that $\xi\rtimes\T$ induces a 
$KK^\Z$-equivalence $\eta$ 
between $(A^\alpha,\check{\alpha})$ 
and $(B^\beta,\check{\beta})$. Since $\xi$ is 
unital, we can choose $\eta$ to be unital as well. 

Since $A^\alpha$ and $B^\beta$
are Kirchberg algebras by part~(3) of 
\autoref{prop:KirAperiodic}, it follows from 
Theorem~4.2.1 in~\cite{Phi_Classif} that there 
exists an isomorphism $\phi\colon A^\alpha\to B^\beta$ such that $KK(\phi)=\eta$. Since $\eta$ is equivariant,
it follows that 
$\phi\circ \check{\alpha}\circ\phi^{-1}$ and $\check{\beta}$ determine the same
class in $KK(B^\beta,B^\beta)$. Since $A$ and $B$ are simple, it follows from part~(1) of \autoref{prop:KirAperiodic} that $\check{\alpha}$ and $\check{\beta}$ are aperiodic. Thus, by Theorem~5 
in~\cite{Nak_aperiodic_2000} there exists a unitary
$w\in \U(B^\beta)$ with $\phi\circ \check{\alpha}\circ\phi^{-1}=\Ad(w)\circ\check{\beta}$. In other words,
$\check{\alpha}$ and $\check{\beta}$ are cocycle 
conjugate. It follows that their dual actions are 
conjugate, and hence $(A,\alpha)\cong (B,\beta)$ as 
desired.


We turn to the last part of the statement.
Fix a unital $KK^\T$-equivalence $\rho\in KK^\T((A,\alpha),(B,\beta))$.
Arguing as in the first part using Baaj-Skandalis duality, 
we obtain a unital $KK^\Z$-equi\-va\-lence
\[\eta\in KK^\Z\big((A^\alpha,\check{\alpha}),(B^\beta,\check{\beta})\big).\]
Denote by $\mathcal{F}(\eta)\in KK(A^\alpha,B^\beta)$ the
$KK$-equivalence that $\kappa$ induces under the forgetful functor 
$KK^\Z\to KK$. 
Similarly, we let $\mathcal{F}(\eta\rtimes\Z)\in KK(A,A)$
denote the $KK$-equivalence that $\kappa\rtimes\Z$ induces, which
can be canonically identified with $\mathcal{F}(\rho)$. 
Using naturality of the functors involved, 
we deduce that the diagram
\begin{align*}
\xymatrix{K_j(A^\alpha) \ar[rr] \ar[d]^{\mathcal{F}(\eta)_j} && K_j(A)\ar[d]^{\mathcal{F}(\rho)_j} \ar[rr] && K_{1-j}(A^\alpha)\ar[d]^-{\mathcal{F}(\eta)_{1-j}}\\
K_j(B^\beta) \ar[rr]&& K_j(B)\ar[rr] && K_{1-j}(B^\beta),}\end{align*}
commutes. Since all vertical maps are isomorphisms and 
$\mathcal{F}(\eta)_0([1_{A^\alpha}])=[1_{B^\beta}]$,
we deduce that $(\Ext_\ast(\alpha),[1_{A^\alpha}])\cong
(\Ext_\ast(\beta),[1_{B^\beta}])$.

We now prove the converse, so assume that $A$ and $B$ satisfy the UCT. 
By 
part~(4) of~\autoref{prop:KirAperiodic},
$A^\alpha$ and $B^\beta$ also satisfy the UCT.
Since $\check{\alpha}$ and $\check{\beta}$ are
approximately inner, an isomorphism $(\Ext_\ast(\alpha),[1_{A^\alpha}])\cong
(\Ext_\ast(\beta),[1_{B^\beta}])$
is equivalent to $\check{\alpha}$ and $\check{\beta}$ having isomorphic
Pimsner-Voiculescu 6-term exact sequences (in a unit-preserving way); 
see the comments before \autoref{df:IsomExt}.
It thus follows from Theorem~2.12 
in~\cite{Mey_more_2020} that $(A^\alpha,\check{\alpha})$
is unitally $KK^\Z$-equivalent to $(B^\beta,\check{\beta})$. Again by 
Baaj-Skandalis duality, it 
follows that $(A,\alpha)$ is unitally $KK^\T$-equivalent
to $(B,\beta)$. This finishes the proof.
\end{proof}

In the context of \autoref{thm:ClassRpKir}, 
the assumption that the diagram (\ref{eqn:Commut}) from
\autoref{df:IsomExt} commutes cannot be dropped, and it is not enough
to have isomorphisms $K_\ast(A^\alpha)\cong K_\ast(B^\beta)$
and $K_\ast(A)\cong K_\ast(B)$.
In the next example, we construct 
a Kirchberg algebra $A$ satisfying the UCT, and two actions
$\alpha,\gamma\colon \T\to \Aut(A)$ with the Rokhlin property, 
such that $A^\alpha\cong A^\gamma$ but $\alpha$ and $\gamma$ are not 
conjugate. In particular, the example shows that an isomorphism 
of the $K^\T$-theory cannot in general be lifted to a $KK^\T$-equivalence.

In preparation for our construction, we introduce some notation.
If $B$ is a unital \ca\ and $\varphi\in\Aut(B)$ is an approximately
inner automorphism, then 
the Pimsner-Voiculescu exact sequence for $\varphi$ reduces to the
short exact sequences
\[\label{eqn:PVsplit}\tag{4.2}\xymatrix{0\ar[r] & K_j(B)\ar[r]^-{K_j(\iota)} & K_j(B\rtimes_\varphi \Z)\ar[r] & K_{1-j}(B)\ar[r]& 0},\]
for $j=0,1$. 
We denote by $\eta_j(\varphi)
\in \Ext(K_{1-j}(B),K_j(B))$ 
the class of the above extension. Note that if $\varphi,\psi\in\Aut(B)$ are 
cocycle conjugate automorphisms, then $\eta_j(\varphi)=\eta_j(\psi)$.

\begin{eg}\label{eg:EqKThyNotEnough}
Set 
\[K=\Q\oplus\bigoplus_{n=1}^\I ((\Z\rtimes_{-1}\Z)\oplus \Z)
\ \ \mbox{ and } \ \ E=\Q\oplus K,\]
and note that $K$ and $E$ are torsion free. 
Set $k_0=(1_{\Q},0,0,\ldots)\in K$
and $e_0=(1_{\Q},k_0)\in E$. 
Using that $\Z\rtimes_{-1}\Z$ is a non-trivial extension of
$\Z$ by $\Z$, fix a non-trivial extension
\[\label{eqn:Ext}\tag{4.3} 0\to K \to E\to K\to 0,\]
and write $\xi\in \Ext(K,K)$ for induced class. 
Since $K$ is torsion free, it follows that the extension in
(\ref{eqn:Ext}) is pure. 
Note that there is also an isomorphism
$E\cong K\oplus K$. (In other words, and this is a crucial
ingredient in the construction, 
the group $E$ can be written in two 
non-equivalent ways as an extension of $K$ by itself.)

Since $K$ is torsion-free,
we may use Elliott's classification of simple 
A$\T$-algebras with real rank zero (see 
the comments before Proposition 3.2.7 in
\cite{Ror_BookClassif}),
to find a simple, unital A$\T$-algebra $C$ satisfying
$K_0(C)\cong K_1(C)\cong K$ with $[1_C]$ corresponding to $k_0\in K$.
Use the case $i=1$ of Theorem~3.1 in~\cite{KisKum_class_1998} 
to find an approximately inner automorphism $\varphi\in \Aut(C)$ 
with $\eta_0(\varphi)=0$ and $\eta_1(\varphi)=\xi\in \Ext(K,K)$.
Then
\[K_0(C\rtimes_\varphi\Z)\cong E \cong 
 K_1(C\rtimes_\varphi\Z).
\]
The proof of Theorem~3.1 in~\cite{KisKum_class_1998} is in fact 
constructive, and the argument used to prove the 
case $i=1$ shows that $\varphi$ can
be chosen to be approximately representable. Indeed, it is shown
in Subsection~3.11 of \cite{KisKum_class_1998} that
there is an increasing sequence $(C_n)_{n\in\N}$ of unital subalgebras of 
$C$ with $C=\varinjlim C_n$, and unitaries
$u_n\in C_n$ satisfying $\varphi=\varinjlim \Ad(u_n)$.
\vspace{0.2cm}

\textbf{Claim:} \emph{$\varphi$ is aperiodic.}
Arguing by contradiction, suppose that there exist
$n\geq 1$ and $u\in\U(C)$ 
such that $\varphi^n=\Ad(u)$. 
Set $v=u\varphi(u)\cdots \varphi^{n-1}(u)\in C$. 
Then $\varphi^{n^2}=\Ad(v)$ 
and $v$ is $\varphi^{n^2}$-invariant.
Moreover, we have
\begin{align*}
\varphi(v)&=
 \varphi(u)\varphi^2(u)\cdots\varphi^{n-1}(u) \varphi^{n}(u)= \varphi(u)\varphi^2(u)\cdots\varphi^{n-1}(u) u\\
&=u\varphi^{-n}\big(\varphi(u)\varphi^2(u)\cdots\varphi^{n-1}(u) \big)=
u\varphi\big(\varphi^{-n}(u)\big)\varphi^2\big(\varphi^{-n}(u)\big)\cdots\varphi^{n-1}(\varphi^{-n}(u)\big)\\
&=u\varphi(u)\cdots\varphi^{n-1}(u)=v,
\end{align*}
and thus $v$ is $\varphi$-invariant. 
Denote by $D$ the twisted crossed product of $C$ by $\varphi$
with respect to the twist induced by $v$.
By Theorem~2.4 in~\cite{OlePed_partially_1986},
the crossed product $C\rtimes_\varphi\Z$ is isomorphic to
the induced algebra $\mathrm{Ind}_{(n^2\Z)^{\perp}}^\T (D)$. 
By compactness of $\T$, this induced algebra is isomorphic
to $C(\T,D)$. In particular, 
\[\label{eqn:KgroupsD}\tag{4.4}
E\cong K_0(C\rtimes_\varphi\Z)\cong 
K_1(C\rtimes_\varphi\Z)\cong K_0(D)\oplus K_1(D).\] 

We proceed to compute the $K$-theory of $D$. 
Set $\K=\K(\ell^2(\Z))$.
By Theorem~3.4 in~\cite{PacRae_twisted_1989}, 
there is an automorphism $\varphi_0$ of 
$C\otimes \K$ whose crossed product is isomorphic to $D$ (this is the so-called
Packer-Raeburn stabilisation trick). An explicit formula
for $\varphi_0$ is given at the beginning of the proof 
on page 301 (see Equation (3.1)), which shows that we 
can choose $\varphi_0$ to be unitarily equivalent to 
$\varphi\otimes\id_{\K}$ (see also the top line on
page 302). It follows that $D$ is isomorphic to 
$(C\otimes\K)\rtimes_{\varphi\otimes\id_\K}\Z\cong 
(C\rtimes_\varphi\Z)\otimes\K$, and thus has the 
same $K$-groups as $C\rtimes_\varphi\Z$. That is,
$K_0(D)\cong K_1(D)\cong E$.

Combining the above with (\ref{eqn:KgroupsD}), 
we deduce that
$E\cong E\oplus E$. This is, however, not the case: the largest 
divisible subgroup of $E$ is $\Q^2$, while the largest divisible subgroup
of $E\oplus E$ is $\Q^4$. This contradiction shows that $\varphi$ is 
aperiodic, proving the claim.
\vspace{0.2cm}

Set $B=C\rtimes_\varphi\Z$ and $\beta=\widehat{\varphi}\colon \T\to\Aut(B)$.
Then $\beta$ has the Rokhlin property by \autoref{thm:DualityRpAr}, since 
$\varphi$ is approximately representable. Moreover, $B$
is unital, separable, nuclear, satisfies the UCT, 
and is simple since $\varphi$ is aperiodic. 
Set $A=B\otimes\OI$ and $\alpha=\beta\otimes\id_{\OI}$. Then
$A$ is a unital Kirchberg algebra satisfying the UCT, and $\alpha$
has the Rokhlin property. Note that 
\[A^\alpha=(C\rtimes_\varphi\Z)^\beta\otimes\OI=C\otimes\OI,\]
and thus $A^\alpha$ 
is the unique unital Kirchberg algebra satisfying the
UCT with $K$-theory given by $K_0(A^\alpha)\cong K_1(A^\alpha)\cong K$, with unit class $k_0$.

We now wish to realize $A$ in a different way as a crossed product
by an approximately representable automorphism, in such a way such that 
the associated Pimsner-Voiculescu exact sequence splits into two
copies of the trivial extension of $K$ by itself. 

Let $\psi_0\in\Aut(\OI)$ be an aperiodic, approximately representable
automorphism of $\OI$ (see, for example, Proposition~3.3 in~\cite{Gar_Kir2}).
Set $\psi=\id_{A^\alpha}\otimes \psi_0$, which we identify with
an aperiodic, approximately representable automorphism of $A^\alpha$.
Since $K_1(\OI)=0$, both classes $\eta_0(\psi_0)$ and $\eta_1(\psi_0)$
are trivial. It follows that the same is true for 
$\eta_0(\psi)=\eta_1(\psi)=0$ in $\Ext(K,K)$. 
The crossed product $A^\alpha\rtimes_\psi\Z$ is therefore a unital 
Kirchberg algebra satisfying the UCT with both $K$-groups isomorphic to 
$E$, and unit class $e_0$. It follows from the classification
of Kirchberg algebras 
(specifically Theorem~4.2.4 in~\cite{Phi_Classif})
that there is an isomorphism $A\cong A^\alpha\rtimes_\psi\Z$. Denote by
$\gamma\colon \T\to\Aut(A)$ the action that $\widehat{\psi}$ induces
on $A$ via this identification. Then $\gamma$ has the Rokhlin property 
because $\psi$ is approximately representable. 
Moreover, 
\[A^\gamma\cong (A^\alpha\rtimes_\psi\Z)^{\widehat{\psi}}=A^\alpha.\]
It follows that $K_\ast^\T(A,\alpha)\cong K^\T_\ast(A,\gamma)$ as groups 
(in fact, $\alpha$ and $\gamma$ have isomorphic
crossed products and fixed point algebras). 
This group isomorphism is automatically an 
isomorphism of $R(\T)$-modules, since the 
action of $R(\T)\cong \Z[x,x^{-1}]$ on 
$K_\ast^\T$
is determined by the dual automorphism, which for 
Rokhlin actions is approximately inner (by 
\autoref{thm:DualityRpAr}) and hence trivial on 
$K$-theory.

We denote by $\iota^\alpha\colon A^\alpha
\to A$ and $\iota^\gamma\colon A^\alpha\to A$ the induced inclusions of
$A^\alpha$ into $A$ as the $\alpha$- and $\gamma$-fixed point algebras, respectively. 
\vspace{0.2cm}

\textbf{Claim:} \emph{there is no 
$KK^\T$-equivalence (unital or otherwise) 
between $(A,\alpha)$ and $(A,\gamma)$.}
Arguing by contradiction, let us assume that there
exists a $KK^\T$-equivalence $\rho\in KK^\T((A,\alpha),(A,\gamma))$.
Using the Baaj-Skandalis duality, we get a $KK^\Z$-equi\-va\-lence
\[\rho\rtimes\T \in KK^\Z\big((A\rtimes_\alpha\T,\widehat{\alpha}),
 (A\rtimes_\gamma\T,\widehat{\gamma})\big).
\]
Since $(A\rtimes_\alpha\T,\widehat{\alpha})\sim_{KK^\Z} (A^\alpha,\check{\alpha})$ and $(A\rtimes_\alpha\T,\widehat{\gamma})\sim_{KK^\Z}
(A^\alpha,\check{\gamma})$ by Takai duality, 
we identify $\rho\rtimes\T$ with a 
$KK^\Z$-equivalence
\[\kappa\in KK^\Z\big((A^\alpha,\check{\alpha}),(A^\alpha,\check{\gamma})\big).\]
Denote by $\mathcal{F}(\kappa)\in KK(A^\alpha,A^\alpha)$ the
$KK$-equivalence that $\kappa$ induces under the forgetful functor 
$KK^\Z\to KK$. Similarly, we let $\mathcal{F}(\kappa\rtimes\Z)\in KK(A,A)$
denote the $KK$-equivalence that $\kappa\rtimes\Z$ induces, which
can be canonically identified with $\mathcal{F}(\rho)$. 
By naturality of all the functors
involved, there is a commutative diagram
\begin{align*}\label{eqn:EquivPV}\tag{4.5}
\xymatrix{0\ar[r] & K_1(A^\alpha)\ar[r]^{\iota^\alpha_\ast} \ar[d]^-{\kappa_\ast} &
K_1(A)\ar[r]\ar[d]^-{(\kappa\rtimes\Z)_\ast} & K_0(A^\alpha)\ar[r] \ar[d]^-{\kappa_\ast} &0\\
0\ar[r] & K_1(A^\alpha)\ar[r]_-{\iota^\gamma_\ast} &
K_1(A)\ar[r] & K_0(A^\alpha)\ar[r]&0,}
\end{align*}
where the vertical maps are all group isomorphisms.
Note that the horizontal short exact sequences are the Pimsner-Voiculescu
sequences associated to $\check{\alpha}$ and $\check{\gamma}$, 
respectively, as in (\ref{eqn:PVsplit}), and thus
their $\mathrm{Ext}$-classes are
$\eta_1(\check{\alpha})$ and
$\eta_1(\check{\gamma})$, respectively. 
Moreover, $\check{\alpha}=\varphi\otimes\id_{\OI}$ and 
under the $KK$-equivalence $C\sim_{KK}C\otimes \OI$
induced by $\C\sim_{KK}\OI$,
the class $\eta_1(\check{\alpha})$ corresponds to $\xi$. 
Similarly, $\check{\gamma}$ is conjugate to 
$\id_{A^\alpha}\otimes\psi_0\in\Aut(A^\alpha\otimes\OI)$, 
which under the $KK$-equivalence $A^\alpha\sim_{KK}A^\alpha\otimes\OI$
induced by $\C\sim_{KK}\OI$, 
the class $\eta_1(\check{\gamma})$ corresponds to $\eta_1(\id_{A^\alpha})=0$.

Commutativity of (\ref{eqn:EquivPV}) implies that $\xi=0$, which is a contradiction.
We conclude 
that $\alpha$ and $\gamma$ are not $KK^\T$-equivalent. In particular, 
$\alpha$ and $\gamma$ are not conjugate.
\end{eg}

The above example shows an interesting phenomenon, worth putting into 
perspective. 
In Example~10.6 
in~\cite{RosSch_kunneth_1986}, the authors construct two circle actions
on commutative \ca s with isomorphic $K^\T$-theory, which are
not $KK^\T$-equivalent. In their example, the 
underlying algebras are not even $KK$-equivalent, so the actions
cannot be $KK^\T$-equivalent. 
As informed to us by Schochet,
there was until now no known example of 
two circle actions on 
the same \ca, with
isomorphic fixed point algebras and 
crossed products, all satisfying the UCT,
and with isomorphic $K^\T$-theory, 
that are not $KK^\T$-equivalent. Our 
construction provides such an example. 

\autoref{thm:ClassRpKir} states that, for Rokhlin
actions on UCT Kirchberg algebras, the class 
Ext($\alpha)$ defined in part~(1) of 
\autoref{thm:PureExt} is a complete invariant up to 
conjugacy. It is then natural to ask for a range result, namely, which Ext-classes arise from 
Rokhlin actions on UCT Kirchberg algebras. 
In the next result, we show that the only possible
obstructions are the ones we obtained in 
\autoref{thm:PureExt}, specifically that the extension
is pure. 

\begin{thm}\label{thm:rangeInv}
Let $K_0$ and $K_1$ be abelian groups, let $k_0\in K_0$,
and let $\mathcal{E}_0\in \Ext(K_0,K_1)$ and $\mathcal{E}_1\in \Ext(K_1,K_0)$
be extensions.
The following are equivalent:
\be\item[(a)] There is a Rokhlin action $\alpha\colon \T\to\Aut(A)$ 
on a unital UCT Kirchberg algebra
$A$ with $(\Ext_\ast(\alpha),[1_{A^\alpha}])\cong (\mathcal{E}_0,\mathcal{E}_1,k_0)$ in the sense of \autoref{df:IsomExt};
\item[(b)] $\mathcal{E}_0$ and $\mathcal{E}_1$ are pure. 
\ee
\end{thm}
\begin{proof}
That (a) implies (b) follows from part~(1) of \autoref{thm:PureExt}.
Assume that (b) holds, and write $\mathcal{E}_0$ and $\mathcal{E}_1$
explicitly as
\[(\mathcal{E}_0) \ \ \xymatrix@C=1.5em{ 0 \ar[r] & K_0 \ar[r]^-{\iota_0}& G_0 \ar[r]^-{\pi_0} &K_1\ar[r]& 0} 
\ \ \ \ \mbox{and } \ \ \ \ 
(\mathcal{E}_1) \ \ \xymatrix@C=1.5em{ 0 \ar[r] & K_1 \ar[r]^-{\iota_1}& G_1 \ar[r]^-{\pi_1} &K_0\ar[r]& 0.}
\]
For $j=0,1$, let $(K_j^{(n)})_{n\in\N}$ be an increasing sequence
of finitely generated subgroups of $K_j$ whose union equals $K_j$.
Without loss of generality, we assume that $k_0$ belongs to $K_0^{(n)}$
for all $n\in\N$. Fix $j=0,1$. Since $\mathcal{E}_j$ is 
pure, for every $n\in\N$ there exists a (necessarily finitely generated)
subgroup $\widetilde{K}_{1-j}^{(n)}$ of $G_j$ such that $\pi_j$
restricts to an isomorphism $\widetilde{K}_{1-j}^{(n)}\cong K_{1-j}^{(n)}$.
We can choose $\widetilde{K}_{1-j}^{(n)}$ so that 
$\widetilde{K}_{1-j}^{(n)}\subseteq \widetilde{K}_{1-j}^{(n+1)}$. 

Let $G_j^{(n)}$ be the subgroup of $G_j$ generated by $\iota_j(K_j^{(n)})$
and $\widetilde{K}_{1-j}^{(n)}$. We denote by $\mathcal{E}_j^{(n)}$ the 
restricted short exact sequence
\[(\mathcal{E}_j^{(n)}) \ \ \ \ \ \ \ \ \xymatrix{ 0 \ar[r] & K_j^{(n)} \ar[r]& G_j^{(n)} \ar[r] &K_{1-j}^{(n)}\ar[r]& 0},\] 
where the maps are the restrictions of the ones for $\mathcal{E}_j$. 
Observe that there are canonical inclusions 
$f_j^{(n)}\colon K_j^{(n)}\to K_j^{(n+1)}$ and 
$h_j^{(n)}\colon G_j^{(n)}\to G_j^{(n+1)}$, making the following diagram
commute:
\begin{align*}\label{eqn:InclExts}\tag{4.6}
\xymatrix{ 0 \ar[r] & K_j^{(n)}\ar[d]_-{f^{(n)}_j} \ar[r]& G_j^{(n)}\ar[d]_-{h^{(n)}_j} \ar[r] &K_{1-j}^{(n)}\ar[d]_-{f^{(n)}_{1-j}}\ar[r]& 0\\
0 \ar[r] & K_j^{(n+1)} \ar[r]& G_j^{(n+1)} \ar[r] &K_{1-j}^{(n+1)}\ar[r]& 0.}
\end{align*}

Fix $n\in\N$. 
Since $\mathcal{E}_j$ is pure and $G^{(n)}_j$ is finitely generated, 
it follows that $\mathcal{E}^{(n)}_j$ is isomorphic to the trivial extension.
Thus there are isomorphisms
\[G_0^{(n)}\cong K_0^{(n)}\oplus K_1^{(n)}\cong G_1^{(n)}.\]
Let $\widetilde{B}^{(n)}$ be a UCT Kirchberg algebra satisfying
\[(K_0(\widetilde{B}^{(n)}),K_1(\widetilde{B}^{(n)}),[1_B])=(K_0^{(n)}, K_1^{(n)}, k_0),\] and use Theorem~4.1.1 in \cite{Phi_Classif}
to find a unital homomorphism $\widetilde{\theta}_n\colon \widetilde{B}_n
\to \widetilde{B}_{(n+1)}$ inducing 
\[f^{(n)}_0\oplus f^{(n)}_1\colon K_0(\widetilde{B}_n)\oplus K_1(\widetilde{B}_n)
 \to K_0(\widetilde{B}_{(n+1)})\oplus K_1(\widetilde{B}_{(n+1)})
\]
at the level of $K$-theory. 
Fix an aperiodic approximately representable automorphism $\Phi$ of $\OI$.
Set $B_n=\widetilde{B}_n\otimes \OI$, $\theta_n=\widetilde{\theta}_n\otimes\id_{\OI}$, 
and $\varphi_n=\id_{\widetilde{B}_n}\otimes \Phi\in\Aut(B_n)$.
Then $\varphi_n$ is approximately representable and aperiodic. Set 
$A_n=B_n\rtimes_{\varphi_n}\Z$ and let $\alpha^{(n)}\colon \T\to\Aut(A_n)$
denote the dual action of $\varphi_n$. Then $A_n$ is a Kirchberg algebra
satisfying the UCT by part~(4) of \autoref{prop:KirAperiodic}, and 
$\alpha^{(n)}$ has the Rokhlin property by \autoref{thm:DualityRpAr}. 
Let $\rho_n\colon (A_n,\alpha^{(n)})\to (A_{n+1},\alpha^{(n+1)})$
be the unital equivariant homomorphism induced by $\theta_n$. 
Since the $K$-theory of $A_n$ is finitely generated, it follows from
part~(3) of \autoref{thm:PureExt} that $\Ext_\ast(\alpha^{(n)})$ is
isomorphic to the trivial extension, and hence
$(\Ext_\ast(\alpha^{(n)}),k_0)\cong (\mathcal{E}_0^{(n)},\mathcal{E}_1^{(n)},k_0)$.

Finally, we set $(A,\alpha)=\varinjlim(A_n,\rho_n,\alpha^{(n)})$.
Then $A$ is a Kirchberg algebra satisfying the UCT, $\alpha$ has the 
Rokhlin property by part~(4) of Theorem~2.5 in~\cite{Gar_CptRok}.
Since $K_j(\rho_n)$ is identified with $h_j^{(n)}\colon G_j^{(n)}\to 
G_j^{(n+1)}$, it follows that $\rho_n$ induces the embedding
$\mathcal{E}_j^{(n)}\hookrightarrow \mathcal{E}_j^{(n+1)}$ from
diagram (\ref{eqn:InclExts}). Since the limit of such diagrams is $\mathcal{E}_j$, we
conclude that 
$(\Ext_\ast(\alpha),[1_{A^\alpha}])=(\mathcal{E}_0,\mathcal{E}_1,k_0)$,
as desired.\end{proof}

The following consequence of \autoref{thm:ClassRpKir} and 
\autoref{thm:rangeInv}
will be needed later.

\begin{cor}\label{cor:UniqueActionOt}
Up to conjugacy there exists a unique circle 
action on $\Ot$ with the Rokhlin property.
\end{cor}

We turn to Theorems~\ref{thmintro:KKeqKir} and 
\ref{thmintro:Emb} of the introduction,
relating Rokhlin actions on arbitrary \ca s with
Rokhlin actions on Kirchberg algebras.
We will need to consider a specific 
automorphism of $\Ot$; its explicit description
will be crucial in the proof of \autoref{thm:EveryRpKKequivKirch}. 

\begin{nota}\label{nota:Psi}
For every $n\geq 1$, let $u_n\in M_n\cong \B(\ell^2(\{0,\ldots,n-1\}))$ be the
unitary given by $u_n(\delta_j)=\delta_{j+1}$ for 
$j=0,\ldots,n-1$, where the subscripts are taken modulo $n$. 
(In particular, $u_1=1\in \C$.)
Let 
$\rho_n\colon M_n\to \Ot$ be a unital
homomorphism. Fix an isomorphism
$\kappa\colon \bigotimes_{n=1}^\I\Ot\to \Ot$, and 
let $\Psi\in\Aut(\Ot)$ be the automorphism
satisfying $\Psi\circ \kappa=\kappa\circ \bigotimes_{n=1}^\I \Ad(\rho_n(u_n))$.
One can check with elementary methods that $\Psi$ is aperiodic.
\end{nota}

\begin{prop}\label{thm:EmbedExactAutom}
Let $B$ be a unital, separable, exact \ca, and let
$\varphi\in\Aut(B)$ be an approximately inner automorphism. Then
there exist a unitary $v\in\Ot$
and a unital, equivariant embedding 
$(B,\varphi)\to (\Ot,\Ad(v)\circ \Psi)$.

Moreover, the unitary $v$ can be chosen so that
there is a unital, equivariant embedding 
$(\Ot,\id_\Ot)\to (\Ot,\Ad(v)\circ \Psi)$.
\end{prop}
\begin{proof} 
Fix a unital embedding $\theta\colon B\to \Ot$. 
Denote by $\pi\colon \ell^\I(\Ot)\to (\Ot)_\I$ the 
canonical quotient map.
Let $(u_n)_{n\in\N}$ be any sequence of unitaries in $B$
satisfying $\varphi(b)=\lim\limits_{n\to\I} \Ad(u_n)(b)$ for all $b\in B$, and set 
$z=\pi\big((\theta(u_n))_{n\in\N}\big)\in (\Ot)_\I$.
Then the unitary $z$ 
satisfies $\varphi(b)=z\theta(b)z^*$ for all $b\in B$.
By the universal property of the crossed product, there exists
a unital embedding 
$\Theta \colon B\rtimes_{\varphi}\Z\to (\Ot)_\I$.

We will show that $\Theta$ lifts to a unital completely
positive map 
$B\rtimes_\varphi\Z\to \ell^\I(\Ot)$ which agrees with 
$\theta$ on the canonical copy of $B$. 
Denote by $u\in  B\rtimes_{\varphi}\Z$ the canonical
unitary implementing $\varphi$
For $n\in\N$, let $\phi_n\colon B\rtimes_\varphi\Z\to M_n(B)$ and $\psi_n\colon M_n(B)\to B\rtimes_\varphi\Z$
be the unital completely positive maps given as follows
\[\phi_n(bu^j)=\begin{cases*}
      \sum\limits_{k=j}^{n-1}\varphi^{-k}(b)e_{k,k-j} & if $0\leq j<n$ \\
      \sum\limits_{k=0}^{j+n-1}\varphi^{-k}(b)e_{k,k-j} & if $-n< j\leq 0$\\
      0 & otherwise,
    \end{cases*} \ \mbox{ and } \ \ \psi_n(b\otimes e_{j,k})=\frac{1}{n}\varphi^k(a)u^{k-j}.
\]
One readily checks that $\psi_n(\phi_n(bv^j))=\frac{n-j}{n}bv^j$ if $|j|< n$, and $0$ otherwise. In particular,
$\psi_n\circ\phi_n$ converges pointwise in norm to 
$\id_{B\rtimes_\varphi\Z}$, and thus
$\Theta\circ\psi_n\circ\phi_n$ converges pontwise in norm 
to $\Theta$.
Since $B\rtimes_\varphi\Z$ is separable, it follows from
Theorem~6 in~\cite{Arv_notes_1977} that the set of 
maps $B\rtimes_\varphi\Z\to (\Ot)_\I$
which have unital completely positive lifts 
into $\ell^\I(\Ot)$ is closed in the 
point-norm topology. In particular, 
it suffices to show that for every $n\in\N$, the map 
$\Theta\circ\psi_n\circ\phi_n$
has a unital completely positive lift which extends $\theta$.

Using nuclearity of $M_n$, let 
$\widetilde{\rho}_n\colon M_n\to \ell^\I(\Ot)$ be any 
unital compeltely positive lift of 
the composition
\[\xymatrix{
M_n\ar[rr]^-{x\mapsto x\otimes 1_B} 
&& M_n(B)\ar[r]^-{\psi_n} & B\rtimes_\varphi\Z\ar[r]^-{\Theta} & (\Ot)_\I.}
\]
Let $\rho_n\colon M_n(B) \to \ell^\I(\Ot)$ be the linear 
map determined by
\[\rho_n(b\otimes e_{j,k})=\widetilde{\rho}_n(e_{j,j})^{1/2}\theta(b)\widetilde{\rho}_n(e_{k,k})^{1/2}\]
for all $b\in B$ and $j,k=0,\ldots,n-1$.
Then $\rho_n$ is unital. We claim that $\rho_n$ is 
positive. Since a positive element in $M_n(B)$ is the 
sum of $n$ elements of the form 
$\sum_{j,k=0}^{n-1} b_j^*b_k\otimes e_{j,k}$
for some $b_0,\ldots,b_{n-1}\in B$, it suffices to show 
that $\rho_n$ preserves positivity of such elements.
Given $b_0,\ldots,b_{n-1}\in B$, set
\[b=\sum_{j,k=0}^{n-1} b_j^*b_k\otimes e_{j,k}\in M_n(B)
  \ \ \mbox{ and } \ \ x=\sum_{j=0}^{n-1}\theta(b_j)\widetilde{\rho}_n(e_{j,j})^{1/2}\in \ell^\I(\Ot).\]
Then $\rho_n(b)=x^*x\geq 0$, as desired. It follows that 
$\rho_n$ is positive, and one shows in a similar way 
that it is completely positive.

Note that $\pi\circ\rho_n=\Theta\circ \psi_n$,
and this $\rho_n\circ\phi_n$ is a lift for 
$\Theta\circ \psi_n\circ\phi_n$ which agrees
with $\theta$ on the canonical copy of $B$. 
Thus there exists a unital completely positive lift
for $\Theta$ which extends $\theta$.

Using Lemma~2.2 in~\cite{KirPhi_embedding_2000},
find a unital embedding 
$\sigma\colon B\rtimes_{\varphi}\Z\to\Ot$ satisfying
$\sigma(b)=\theta(b)$ for all $b\in B$, and set $w=\sigma(u)$. 
Consider the following embeddings, where 
the first one is the restriction of $\sigma$ to $B$:
\[(B,\varphi)\hookrightarrow (\Ot,\Ad(w))\hookrightarrow (\Ot\otimes\Ot,\Ad(w)\otimes\Psi).\]
Note that $\Ad(w)$ is unitarily equivalent to $\id_{\Ot}$,
and that $\id_{\Ot}\otimes \Psi$ is conjugate to $\Psi$.
It follows that
$\Ad(w)\otimes\Psi$ is unitarily equivalent to 
$\id_\Ot\otimes\Psi$
of $\Ot\cong \Ot\otimes\Ot$. Since $\id_\Ot\otimes\Psi$ is conjugate to $\Psi$,
and thus there is a unital, equivariant embedding
$(B,\varphi)\to (\Ot,\Ad(v)\circ\Psi)$ for some unitary $v\in\Ot$.

Since $(\Ot,\id_{\Ot})$ embeds unitally into $(\Ot,
\Psi)$, it follows that is also embeds unitally into 
$(\Ot\otimes\Ot,\Ad(w)\otimes\Psi)$. Thus the last 
part of the statement also follows from the construction
above.
\end{proof}

\begin{cor}\label{thm:EmbedExact}
Let $A$ be a unital, separable, exact \ca, and let 
$\alpha\colon\T\to\Aut(A)$ be an action with the \Rp.
Denote by $\gamma\colon \T\to\Aut(\Ot)$ the 
unique action with the Rokhlin property 
(see \autoref{cor:UniqueActionOt}).
Then there exists a unital, equivariant embedding 
$(A,\alpha) \hookrightarrow (\Ot,\gamma)$.
\end{cor}
\begin{proof} Note that $A^\alpha$ is unital, separable and 
exact, and that the predual automorphism $\check{\alpha}\in\Aut(A^\alpha)$
is approximately inner by \autoref{thm:DualityRpAr}.
Use \autoref{thm:EmbedExactAutom} to find a unitary 
$v\in\Ot$ and a unital embedding 
\[\label{eqn:Emd}\tag{4.7}\iota\colon (A^\alpha,\check{\alpha})\to (\Ot,\Ad(v)\circ\Psi).\] 

Set $\Phi=\Ad(u)\circ\Psi$. 
Since $\Psi$ is aperiodic, the same is true for 
$\Phi$. 
It follows from parts~(3) and~(4) of 
\autoref{prop:KirAperiodic} that $\Ot\rtimes_\Phi \Z$
is a Kirchberg algebra satisfying the UCT, and it
has trivial $K$-theory by the Pimsner-Voiculescu
exact sequence. Thus $\Ot\rtimes_\Phi\Z$ is isomorphic to $\Ot$. Moreover, since $\Psi$ is
approximately representable by construction, the 
same is true for $\Phi$. 
By \autoref{thm:DualityRpAr}, it follows that 
the dual action $\widehat{\Phi}$ has the \Rp, and 
thus $\widehat{\Phi}$ is conjugate to $\gamma$ 
by \autoref{cor:UniqueActionOt}.

Applying crossed products to (\ref{eqn:Emd}),
and identifying $(A^\alpha\rtimes_{\check{\alpha}}\Z,\widehat{\check{\alpha}})$ with $(A,\alpha)$ 
(see \autoref{thm:RpisDual}), and 
identifying $(\Ot\rtimes_\Phi \Z,\widehat{\Phi})$
with $(\Ot,\gamma)$ as in the paragraph above, 
we obtain a unital homomorphism
$\widetilde{\iota}\colon (A,\alpha)\to (\Ot,\gamma)$.
Finally, $\widetilde{\iota}$ is injective because 
so is $\iota$ and $\Z$ is amenable. 
\end{proof}

\begin{rem}
We point out that it is possible to give a simpler
proof of \autoref{thm:EmbedExact} which avoids \autoref{thm:EmbedExactAutom}, by applying 
Lemma~4.7 in~\cite{AraKub_compact_2017} to any 
embedding $A\hookrightarrow \Ot$. This argument 
even has
the advantage of working for arbitrary compact groups.
We have taken a more indirect approach because
\autoref{thm:EmbedExact} by itself is not enough to prove 
\autoref{thm:EveryRpKKequivKirch}, 
and in its proof we will need to use 
\autoref{thm:EmbedExactAutom} instead.
\end{rem}

Recall that a homomorphism $\pi\colon A\to B$
is said to be \emph{full} if for every $a\in A$
with $a\neq 0$, the ideal in $B$ 
generated by $\pi(a)$ is all of $B$. 
If $\varphi\in\Aut(A)$ and $\psi\in\Aut(B)$ satisfy
$\psi\circ\pi=\pi\circ\varphi$, then the 
$\T$-equivariant
homomorphism $\widehat{\pi}\colon A\rtimes_\varphi\Z
\to B\rtimes_\psi\Z$ induced by $\pi$ is full if and 
only if so is $\pi$.

The following easy observation will be used in the 
proof of \autoref{thm:EveryRpKKequivKirch}.

\begin{rem}\label{rem:FullDirLimSimple}
If $(A_n,\pi_n)_{n\in\N}$ is an inductive system and 
$\pi_n$ is full for every $n\in\N$, then 
$\varinjlim(A_n,\pi_n)$ is simple.
\end{rem}

We are ready to show that every Rokhlin action
on a separable, nuclear \ca\ is $KK^\T$-equivalent 
to a Rokhlin action on a Kirchberg algebra. Our 
proof roughly follows arguments of Kirchberg,
specifically as in Chapter~11 
of~\cite{Kir_Classif}. However, we need to make 
the choices carefully and perform some 
computations in equivariant $KK^\T$-theory 
using the 
\emph{continuous} Rokhlin property from \cite{Gar_Kir2}, in order to guarantee that 
the resulting action on the Kirchberg algebra
has the Rokhlin property. (We do not seem to 
be able to guarantee this using the Cuntz-Pimsner
construction from Theorem~2.1 in~\cite{Mey_classification_2019}.)

\begin{thm}\label{thm:EveryRpKKequivKirch}
Let $A$ be a separable, nuclear \uca, and let
$\alpha\colon \T\to\Aut(A)$ be an action with the \Rp. 
Then there exist a unital Kirchberg algebra $D$ 
and a circle action $\delta\colon \T\to\Aut(D)$
with the \Rp, 
such that $(A,\alpha)\sim_{KK^\T}(D,\delta)$ unitally.
\end{thm}
\begin{proof}
Let $v\in \Ot$ be a unitary as in the conclusion of \autoref{thm:EmbedExact}, and set $\Phi=\Ad(v)\circ\Psi$. Fix
unital embeddings 
\[\tau\colon (\Ot,\id_\Ot)\hookrightarrow 
(\Ot,\Phi) \ \ \mbox{ and } \ \ 
\rho\colon (A^\alpha,\check{\alpha})\to (\Ot,\Phi).\]

We prove the theorem in two steps. First,
assume that there is a unital equivariant embedding
$\sigma\colon (\Ot,\Phi)\to (A^\alpha,\check{\alpha})$.
Let $s_1,s_2\in \mathcal{O}_2$ be 
the canonical generating isometries, and note that
$\tau(s_1)$ and $\tau(s_2)$ are $\Phi$-invariant. 
Set $t_j=\sigma(\tau(s_j))$ for $j=1,2$. Then $t_j$
is an isometry, $\check{\alpha}(t_j)=t_j$, and $t_1t_1^*+t_2t_2^*=1$. 
 Define a unital equivariant map
$\pi\colon (A^\alpha,\check{\alpha})\to (A^\alpha,\check{\alpha})$
by 
\[\pi(a)=t_1\sigma(\rho(a))t_1^*+t_2at_2^*\] for all $a\in A^\alpha$. Then $\pi$ is full. Indeed, if 
$a\in A\setminus\{0\}$, then $\rho(a)\neq 0$ and 
thus pure infiniteness of $\Ot$ implies that
there exist $x,y\in\Ot$ with $x\rho(a)y=1_\Ot$.
Then 
\[\sigma(x)t_1^*\pi(a)t_1\sigma(y)=\sigma(x\rho(a)y)=1,\]
and thus $\pi(a)$ generates $A$ as an ideal, as desired.
\vspace{0.2cm}

\textbf{Claim~1:} \emph{We have $KK^\Z(\pi)=KK^\Z(\id_{A^\alpha})$.} 
By Theorem~F in the introduction of~\cite{Gar_Kir2}
(see also Corollary~3.10 there), 
the action $\gamma\colon\T\to\Aut(\Ot)$ has the 
continuous Rokhlin property, and thus it is
unitally $KK^\T$-equivalent to 
$\big(C(\T)\otimes\Ot,\texttt{Lt}\otimes\id_{\Ot}\big)$ by Theorem~C of~\cite{Gar_Kir2}.
By Baaj-Skandalis duality, and since the dual of 
$\gamma$ is $KK^\Z$-equivalent to $\Phi$, 
we deduce that $(\Ot,\Phi)$ is $KK^\Z$-equivalent
to $(\Ot,\id_{\Ot})$. Thus
\[KK^\Z\big((A^\alpha,\check{\alpha}),(\Ot,\Phi)\big)=0.\]
Since $\sigma\circ\rho$
factors through $(\Ot,\Phi)$, it must be $KK^\Z$-trivial.
We conclude that $KK^\Z(\pi)=KK^\Z(\id_{A^\alpha})$. 
\vspace{0.2cm}

We deduce from the claim above that
$\pi$ is a $KK^\Z$-equivalence, and that 
$\widehat{\pi}$ is a $KK^\T$-equivalence.
Set $(C_0,\theta_0)=\varinjlim \big((A^\alpha,\check{\alpha}),\pi\big)$.
Then $C_0$ is simple because $\pi$ is full 
(see \autoref{rem:FullDirLimSimple}), 
and it is nuclear, unital
and separable because so is $A^\alpha$. 
Denote by $\Pi\colon (A^\alpha,
\check{\alpha})\to (C_0,\theta_0)$ the canonical equivariant map into the 
limit. 
\vspace{0.2cm}

\textbf{Claim~2:} \emph{$\Pi$ is a $KK^\Z$-equivalence.}
(This is not immediate from Claim~1, 
since $KK^\Z$ 
is not a continuous functor.)
By the second
paragraph on page 287 of \cite{MeyNes_homological_2010},
it suffices to show that the 
class
$\mathcal{F}(\Pi)\in KK(A^\alpha, C_0)$ that $\Pi$ 
induces under the forgetful functor $\mathcal{F}$,
is a $KK$-equivalence.
This is contained in Kirchberg's proof (specifically
Sections~11.2 and 11.3 in~\cite{Kir_Classif}) there),
so the claim follows.

\vspace{0.2cm}

Set $(C,\theta)=(C\otimes\OI,\theta_0\otimes\id_{\OI})$. Then $C$
is a unital Kirchberg algebra, and 
$(C,\theta)\sim_{KK^\Z}(A^\alpha,\check{\alpha})$ unitally, since
$(\OI,\id_{\OI})\sim_{KK^\Z}(\C,\id_\C)$ unitally.
\vspace{0.2cm}

\textbf{Claim~3:} \emph{$\theta$ is aperiodic and approximately representable.}
It suffices to prove the claim for $\theta_0$.
Since $(C_0,\theta_0)$ is the direct limit
of $(A^\alpha,\check{\alpha})$ and $\check{\alpha}$
is approximately representable, it follows that so is $\theta_0$.
Note that there is an isomorphism 
$C_0=\varinjlim(A,\widehat{\pi})$. 
Let
$\widehat{\pi}\colon (A,\alpha)\to (A,\alpha)$ be the 
map induced by $\pi$.
Then $\widehat{\pi}$ is full, because so is $\pi$. 
It therefore follows
from \autoref{rem:FullDirLimSimple}
that $C_0$ is simple. By part~(1) 
of \autoref{prop:KirAperiodic}, 
we deduce that $\theta_0$ is aperiodic. 
\vspace{0.2cm}

Set $D=C\rtimes_\theta\Z$ and $\delta=\widehat{\theta}$. 
Since $C$ is a Kirchberg algebra and $\theta$ is 
aperiodic, it follows from part~(3)
of \autoref{prop:KirAperiodic} that $D$ is a unital Kirchberg algebra.
Moreover, $\delta$ has the Rokhlin property by \autoref{thm:DualityRpAr},
since $\theta$ is approximately representable by
Claim~3. 
Finally, by Claim~2, 
the equivariant map $\Pi\otimes \id_{\OI}$ 
is a unital $KK^\Z$-equivalente $(A^\alpha,\check{\alpha})\sim_{KK^\Z}
(C,\theta)$. By taking crossed products, we obtain a unital
$KK^\T$-equivalence $(A,\alpha)\sim_{KK^\T}(D,\delta)$.
This proves the theorem in the case that 
there is a unital equivariant embedding
$\sigma\colon (\Ot,\Phi)\to (A^\alpha,\check{\alpha})$.

We turn to the general case.
Let $\mathcal{O}_\I^{(0)}$ be the unique unital UCT Kirchberg algebra which is $KK$-equivalent to $\OI$ and whose class
of the unit is zero.
\vspace{0.2cm}

\textbf{Claim 4:} \emph{there exist an aperiodic automorphism $\varphi\in\Aut(\mathcal{O}_\I^{(0)})$
and a unital embedding $\iota \colon \Ot\hookrightarrow \mathcal{O}_\I^{(0)}$
such that $\varphi\circ \iota=\iota\circ \Phi$.} 
Fix an injective homomorphism $\nu\colon \Ot\to \OI$,
and set $p=1_\OI-\nu(1)$, which is a projection in 
$\OI$.
Adopt the notation of \autoref{nota:Psi}, 
in particular we will use the unital homomorphisms 
$\rho_n\colon M_n\to \Ot$ and 
the permutation unitaries $u_n\in M_n$.
For $n\in\N$, set 
$w_n=\nu(\rho_n(u_n))+p$,
which is a unitary in $\OI$.
Consider the automorphism 
$\widetilde{\psi}= 
\bigotimes_{n=1}^\I\Ad(w_n)$ of 
$\bigotimes_{n=1}^\I\OI$, and note that
the restriction of $\widetilde{\psi}$ to the 
invariant subalgebra 
$\bigotimes_{n=0}^\I\nu(\Ot)$ can be identified
with $\bigotimes_{n=1}^\I\Ad(\rho_n(u_n))$.
Fix an isomorphism $\widetilde{\kappa}\colon\bigotimes_{n=1}^\I\OI\to \OI$,
and let $\mu\colon \Ot\to\OI$ be the following composition 
\[\xymatrix{\Ot\ar[r]^-{\kappa^{-1}}_-{\cong} &
 \bigotimes_{n=1}^\I \Ot\ar[rr]^-{\bigotimes\limits_{n=1}^\I\nu} && \bigotimes_{n=1}^\I\OI\ar[r]^-{\widetilde{\kappa}}_-{\cong} & \OI.}
\]
Set $w=\mu(v)$ and 
$\psi=\widetilde{\kappa}\circ\Ad(w)\circ \widetilde{\psi}\circ\widetilde{\kappa}^{-1}\in\Aut(\OI)$.
Set $e=\mu(1)$, which is 
a projection in $\OI$ satisfying $\psi(e)=e$
and $[e]=0$ in
$K_0(\OI)\cong\Z$. In particular, there is an
injective homomorphism $j\colon 
\mathcal{O}_{\I}^{(0)}\to \OI$ whose image is 
$e\OI e$. 
Note that $\psi$ leaves the image of 
$j$ invariant.
Then $\mu(v)$ is a unitary in $e\OI e$,
and we let $\varphi\in\Aut(\mathcal{O}_{\I}^{(0)})$ be given by
$\varphi =j^{-1}\circ  \psi\circ j$. 
Let $\iota\colon \Ot\to \mathcal{O}_{\I}^{(0)}$
be the unital homomorphism 
given by $\iota=j^{-1}\circ \mu$.
One readily checks that $\varphi$ and $\iota$
satisfy the conditions in the claim.
\vspace{0.2cm}

\textbf{Claim~5:} \emph{any automorphism
$\varphi\in\Aut(\mathcal{O}_\I^{(0)})$ satisfying the conclusion of Claim~4 satisfies
$(\mathcal{O}_\I^{(0)},\varphi)\sim_{KK^\Z}
(\C,\id_\C)$.}
Since 
$(\mathcal{O}_\I^{(0)},\id_{\mathcal{O}_\infty^{(0)}})$
is $KK^\Z$-equivalent to $(\C,\id_\C)$, it suffices
to show that 
$(\mathcal{O}_\I^{(0)},\varphi)$ is $KK^\Z$-equivalent to $(\mathcal{O}_\I^{(0)},\id_{\mathcal{O}_\infty^{(0)}})$.
To prove this, fix an aperiodic, approximately representable
automorphism $\phi_0\in\Aut(\OI)$ (see, for 
example Proposition~3.3 in~\cite{Gar_Kir2}),
and identify $\phi_0\otimes\id_{\OI}$ with 
an automorphism of $\mathcal{O}_\I^{(0)}\cong
\OI\otimes \mathcal{O}_\I^{(0)}$.
Then 
$KK(\varphi)=KK(\id_{\mathcal{O}_\infty^{(0)}})
=KK(\phi)$. By Theorem~9 
in~\cite{Nak_aperiodic_2000}, it follows that 
$\varphi$ and $\phi$ are cocycle conjugate,
and thus the dual actions $\widehat{\varphi}$ and $\widehat{\phi}$ are conjugate. 
Since $\phi$ is approximately representable, 
it follows from \autoref{thm:DualityRpAr}
that $\widehat{\varphi}$ has the 
Rokhlin property. By Corollary~3.10 
in~\cite{Gar_Kir2} (see also the comments
immediately after it), we deduce that
$\widehat{\varphi}$ has the \emph{continuous}
Rokhlin property, and thus there is a
$KK^\T$-equivalence
\[\big(C(\T)\otimes\mathcal{O}_\infty^{(0)},\texttt{Lt}\otimes\id_{\mathcal{O}_\infty^{(0)}}\big)\sim_{KK^\T}(\mathcal{O}_\infty^{(0)}\rtimes_{\varphi}\Z,\widehat{\varphi})\]
by Theorem~C of~\cite{Gar_Kir2}. Taking crossed
products and using Takai duality, we deduce 
that $(\mathcal{O}_\I^{(0)},\id_{\mathcal{O}_\infty^{(0)}})\sim_{KK^\Z}
(\mathcal{O}_\I^{(0)},\varphi)$, as desired.
\vspace{0.2cm}

Let $\varphi$ be an automorphism satisfying the 
conclusion of Claim~4. Then clearly 
there is a unital equivariant embedding of
$(\Ot,\Phi)$ into 
$(A^\alpha\otimes \mathcal{O}_\I^{(0)},\check{\alpha}\otimes\varphi)$.
Using the first part of this proof, let $\delta_0\colon
\T\to\Aut(D_0)$ be an action with the Rokhlin property
on a unital Kirchberg algebra $D_0$ such that
$(D_0,\delta_0)$ is $KK^\T$-equivalent to the 
dual system of
$(A^\alpha\otimes \mathcal{O}_\I^{(0)},\check{\alpha}\otimes\varphi)$. 
Said dual system is $KK^\T$-equivalent to 
$(A,\alpha)$ by Claim~5, and thus
there is a 
$KK^\T$-equivalence 
$\eta\in KK^\T\big((A,\alpha),(D_0,\delta_0)\big)$. Since $\delta_0$ has the Rokhlin
property, by \autoref{cor: FixedPtCP}
there is a canonical isomorphism
$K_0(D_0\rtimes_{\delta_0}\T)\cong 
K_0(D_0^{\delta_0})$. Combining this with the 
natural
group isomorphism $K_0^\T(D_0,\delta_0)\cong 
K_0(D_0\rtimes_{\delta_0}\T)$ given by 
Julg's theorem, it follows that
there is a natural identification
$K_0^\T(D_0,\delta_0)\cong K_0(D_0^{\delta_0})$. Since $D_0^{\delta_0}$ is a Kirchberg algebra
by part~(3) of \autoref{prop:KirAperiodic},
any element of its $K_0$-group can be realized by a 
projection in $D_0^{\delta_0}$. 
It follows that there exists a projection $q\in D_0^{\delta_0}$
such that $[1_A]\times \eta=[q]$ in 
$K_0^\T(D_0,\delta_0)$.
Set $D=qD_0q$ and let $\delta\colon\T\to\Aut(D)$
denote the restriction of $\delta_0$ to $D$. 
Then $D$ is a unital Kirchberg algebra, and 
$\delta$ has the Rokhlin property by part~(3)
of Theorem~2.5 in~\cite{Gar_CptRok}.
Finally, it is clear that $\eta$ implements a unital
$KK^\T$-equivalence $(A,\alpha)\sim_{KK^\T}(D,\delta)$,
as desired.
\end{proof}


\providecommand{\bysame}{\leavevmode\hbox to3em{\hrulefill}\thinspace}
\providecommand{\MR}{\relax\ifhmode\unskip\space\fi MR }
\providecommand{\MRhref}[2]{%
  \href{http://www.ams.org/mathscinet-getitem?mr=#1}{#2}
}
\providecommand{\href}[2]{#2}

\end{document}